\documentclass[10pt,a4paper]{amsart}
\usepackage[a4paper,margin=1.8cm]{geometry}
\input{glyphtounicode}
\usepackage[T1]{fontenc}
\usepackage{lmodern}
\usepackage[utf8]{inputenc} 
\usepackage[english]{babel}
\usepackage{lmodern}
\usepackage{microtype}
\usepackage{setspace}
\allowdisplaybreaks

\usepackage{amsmath,amssymb,amsfonts,amsthm}
\usepackage{mathtools}
\usepackage{mathrsfs}
\usepackage{stmaryrd}
\usepackage{dsfont}
\usepackage{esint}

\numberwithin{equation}{section}

\usepackage{graphicx}
\usepackage{float}
\usepackage{array}
\usepackage{booktabs}
\usepackage{diagbox}
\usepackage{caption}
\usepackage{subcaption}

\usepackage[dvipsnames]{xcolor}
\usepackage{tikz}
\usetikzlibrary{
  arrows.meta,
  backgrounds,
  calc,
  decorations.markings,
  positioning
}
\usepackage{pgfplots}
\pgfplotsset{compat=1.13}

\newif\ifdraft
\draftfalse

\usepackage[normalem]{ulem}
\usepackage{comment}

\ifdraft
  \usepackage[colorinlistoftodos]{todonotes}
  
  \newcommand{\todoY}[1]{\todo[color=yellow!30]{#1}}
  \newcommand{\marginlabel}[1]{%
    \marginpar{\raggedleft\hspace{0pt}\tiny\textcolor{red}{#1}}%
  }
\else
  
  \newcommand{\todoY}[1]{}
  \newcommand{\marginlabel}[1]{}
\fi

\usepackage{url}
\usepackage[
  colorlinks=true,
  linkcolor=black,
  citecolor=black,
  urlcolor=black,
  bookmarksnumbered=true,
  hyperfootnotes=false
]{hyperref}

\theoremstyle{plain}
\newtheorem{lemma}{Lemma}[section]
\newtheorem{proposition}[lemma]{Proposition}
\newtheorem{theorem}[lemma]{Theorem}
\newtheorem{corollary}[lemma]{Corollary}

\theoremstyle{definition}

\theoremstyle{remark}
\newtheorem{remark}[lemma]{Remark}

\newcommand{\C}{\mathbb C}

\newcommand{\R}{\mathbb R}

\DeclareMathOperator{\supp}{supp}
\DeclareMathOperator{\rank}{rank}

\AtBeginDocument{%
  \setlength{\abovedisplayskip}{1.5pt plus 1pt minus 1pt}%
  \setlength{\belowdisplayskip}{1.5pt plus 1pt minus 1pt}%
  \setlength{\abovedisplayshortskip}{1pt plus 1pt minus 1pt}%
  \setlength{\belowdisplayshortskip}{1pt plus 1pt minus 1pt}%
}

\newcommand\restr[2]{%
  \left.\kern-\nulldelimiterspace
  #1
  \vphantom{\big|}
  \right|_{#2}
}

\title[Kalman Structure and Observability]{Kalman Structure and Observability for Transport Systems}

\author[Y. Simpore ]{%
Yacouba Simpore
}
\date{}
\begin{document}
\begin{abstract}
We study observability and controllability for constant-coefficient first-order hyperbolic systems on the real line when only part of the state is observed or controlled. Even when the Kalman rank condition holds, the usual \(L^2(\mathbb R)^N\)-observability estimate may fail because some components are detected only through the dynamics.

We show that the Kalman structure determines the appropriate
observability estimate. A component that becomes visible after \(k\) algebraic steps is measured at low Fourier frequencies with a weight of order \(|\xi|^{2k}\) in the Fourier variable. This yields a natural Kalman-adapted observation space for the system.

Assuming that the transport matrix \(A\) is diagonalizable over \(\mathbb R\), we prove whole-line observability in this space for every positive time and derive exact controllability in the corresponding dual space. If, in addition, \(A\) has no zero eigenvalue and the observation set satisfies the one-dimensional bounded-gap condition, we establish localized observability whenever the observation time exceeds the natural characteristic travel-time threshold. Separately, in the whole-line setting, we extend the construction to transport matrices \(A\) with real spectrum and possible nontrivial Jordan blocks.
\end{abstract}
\maketitle

\begingroup
\renewcommand\thefootnote{}
\footnotetext{%
\textsuperscript{} Chair for Dynamics, Control, Machine Learning and Numerics,
Department of Mathematics, Friedrich-Alexander-Universit\"{a}t Erlangen--N\"{u}rnberg,
Cauerstrasse~11, 91058 Erlangen, Germany.

\textsuperscript{} Universit\'e Yembila Abdoulaye TOGUYENI,
Burkina Faso.
\textit{E-mail address:}
simpore.yacouba@fau.de.}
\endgroup

\section{Introduction}
\label{sec:introduction}

For hyperbolic systems, partial observability and controllability depend on
the interaction between propagation and algebraic coupling. Information travels
along characteristics, so a localized observation can detect a component only
when the corresponding trajectories meet the observation region. At the same
time, if the observation or the control acts only on part of the state, the
remaining components can be recovered only through the couplings generated by
the dynamics.

On the real line, the algebraic recovery of unobserved directions has a
specific low-frequency obstruction. In contrast with bounded domains, there
are no boundary reflections, compactness effects, or return mechanisms. Even
in the global case, where the observation region is the whole line and no
geometric obstruction remains, the recovery of indirectly observed directions
may fail to be uniform as the Fourier frequency tends to zero. At the Fourier level, this loss is measured by the finite-time observation Gramian, defined below. It leads to a weaker, frequency-dependent observability scale, whose
precise form is identified below.

We study this phenomenon for constant-coefficient first-order hyperbolic
systems on the real line with partial internal observation and the
corresponding partial internal control,
\begin{equation}
\label{eq:controlled-system}
\begin{cases}
y_t+A y_x = B\chi_{\mathcal O}u,
& (t,x)\in (0,T)\times\mathbb R,\\
y(0,x)=y_0(x),
& x\in\mathbb R.
\end{cases}
\end{equation}
Here \(A\in\mathbb R^{N\times N}\), \(B\in\mathbb R^{N\times m}\), with
\(1\leq m\leq N\), and \(\mathcal O\subset\mathbb R\) is the control region.
The matrix \(B\) determines the directly actuated subspace
\(\operatorname{Ran}B\), while \(\chi_{\mathcal O}\) accounts for spatial
localization. When \(\operatorname{rank}B<N\), some directions are not
directly actuated and can be reached only through the transport coupling.

The corresponding adjoint system is
\begin{equation}
\label{adjoint-system}
\begin{cases}
q_t+A^\top q_x=0,
& (t,x)\in(0,T)\times\mathbb R,\\
q(T,x)=q_T(x),
& x\in\mathbb R,
\end{cases}
\end{equation}
and the measured quantity is \(B^\top q\) on
\((0,T)\times\mathcal O\). Thus only the components seen through \(B^\top\) are directly observed.
The relevant algebraic assumption is the Kalman rank condition: there exists
an integer \(K\geq 0\) such that
\[
\operatorname{rank}(B,AB,\ldots,A^K B)=N.
\]
The smallest such integer is called the Kalman index.
In finite dimension, this condition encodes the recovery of unobserved or
uncontrolled directions through successive applications of the dynamics; see,
for instance, \cite{Kalman1963,Seidman}. In the present PDE setting, this
algebraic visibility becomes frequency-dependent. The next subsection
identifies the corresponding low-frequency obstruction.

By the Hilbert Uniqueness Method, observability estimates for
\eqref{adjoint-system} imply controllability results for
\eqref{eq:controlled-system}; see, e.g.,
\cite{Lions1988,TucsnakWeiss2009}. Because of the low-frequency degeneracy,
the natural observability estimate is not formulated in the standard
\(L^2(\mathbb R)^N\) scale. Instead, the correct space is a
Kalman-adapted observation space \(\mathcal Y_K^-\), whose norm weights each
Kalman layer according to its algebraic depth. The corresponding dual
controllability space is denoted by \(\mathcal Y_K^+\). These spaces are
defined in Section~\ref{sec:main-results}.

Throughout the paper we use the Fourier convention
\[
\widehat f(\xi)=\frac{1}{\sqrt{2\pi}}\int_{\mathbb R} f(x)e^{-ix\xi}\,dx,
\quad
\widehat{\partial_x f}(\xi)=i\xi\widehat f(\xi).
\]
The notation \(I_d\) denotes the identity matrix of size \(d\), \(^{\top}\)
denotes transpose, and \(^{*}\) denotes the Hermitian adjoint after
complexification. Inequalities between Hermitian matrices are understood in the
Loewner order, equivalently in the sense of quadratic forms.

\subsection{Low-frequency Kalman obstruction}

We first isolate the obstruction in the global case
\(\mathcal O=\mathbb R\). Then Fourier modes
are decoupled, and any failure of standard \(L^2(\mathbb R)^N\)-observability
is purely algebraic and frequency-dependent. Writing \(s=T-t\), the
observation of the adjoint solution is described, frequency by frequency, by
the finite-time Gramian
\begin{equation}
\label{eq:intro-gramian}
\mathcal G_T(\xi)
=
\int_0^T
e^{-is\xi A}
BB^\top
e^{is\xi A^\top}\,ds .
\end{equation}
A uniform observability estimate in \(L^2(\mathbb R)^N\) would require
\[
\mathcal G_T(\xi)\geq \theta I_N,
\quad \theta>0,
\]
with \(\theta\) independent of \(\xi\). At zero frequency,
\[
\mathcal G_T(0)=TBB^\top .
\]
Thus there is no obstruction at \(\xi=0\) when \(\operatorname{rank}B=N\).
The difficulty appears precisely in the indirect case
\(\operatorname{rank}B<N\), or equivalently
\(\ker B^\top\neq\{0\}\). Then \(\mathcal G_T(0)\) detects only the directions
seen directly by \(B^\top\). For each fixed \(\xi\neq0\), the Kalman coupling may reveal the missing directions, but the strength of this recovery degenerates as \(\xi\to0\). Thus the obstruction is not the absence of Kalman coupling, but the lack of
uniformity of this coupling at low frequencies. The finite-time Gramian records this degeneracy and
selects the weighted scale in which the indirectly observed components can still be measured.

The order of this degeneracy is algebraic. A direction \(v\in\mathbb R^N\) is
first visible at depth \(k\) if
\[
B^\top v=B^\top A^\top v=\cdots
=B^\top(A^\top)^{k-1}v=0,
\quad
B^\top(A^\top)^k v\neq0 .
\]
Then the first nonzero term in the expansion of
\(B^\top e^{is\xi A^\top}v\) appears at order \(k\). The observed amplitude is
of order \(|\xi|^k\), and the observed energy is of order \(|\xi|^{2k}\).

This mechanism is already visible in the one-dimensional wave equation
\[
u_{tt}-u_{xx}=0.
\]
Set
\[
v=u_t,\quad w=u_x,\quad U=(v,w)^\top .
\]
Then
\[
U_t+AU_x=0,
\quad
A=
\begin{pmatrix}
0&-1\\
-1&0
\end{pmatrix}.
\]
Since \(A=A^\top\), the direct and adjoint transport dynamics coincide. If only
the velocity \(v\) is observed, then \(B=e_1\) and \(B^\top U=v\). For an
adjoint terminal datum \(U_T=(0,f)^\top\), lying entirely in the unobserved
strain component, the observed Fourier mode is
\[
\widehat v(t,\xi)
=
-i\sin((T-t)\xi)\widehat f(\xi).
\]
Hence, as \(\xi\to0\),
\[
\int_0^T|\widehat v(t,\xi)|^2\,dt
=
\left(\frac{T^3}{3}|\xi|^2+O(|\xi|^4)\right)
|\widehat f(\xi)|^2 .
\]
Thus the unobserved component is not lost, but it is detected only with
strength \(|\xi|^2\) at the energy level. This is the first nontrivial Kalman
loss: one algebraic interaction produces one power of \(|\xi|\) in amplitude,
and therefore two powers in the observed energy.

\subsection{Kalman-adapted Fourier scale and main contributions}

The results involve two structures: an algebraic one, given by the Kalman filtration, and a geometric one, encoded by the observation set. Let \(K\) be
the Kalman index. We associate with \((A,B)\) the filtration
\[
\mathcal K_k=\operatorname{Ran}(B,AB,\ldots,A^kB),
\quad
\mathcal K_{-1}=\{0\},
\]
and the orthogonal layers
\[
\mathcal V_k=\mathcal K_k\cap \mathcal K_{k-1}^{\perp},
\quad k=0,\ldots,K.
\]
The layer \(\mathcal V_0=\operatorname{Ran}B\) is directly observed. For
\(k\geq1\), directions in \(\mathcal V_k\) become visible only after \(k\)
algebraic interactions. Near \(\xi=0\), the Gramian detects this layer with
energy of order \(|\xi|^{2k}\). Away from zero frequency the loss disappears,
which leads to the weights
\[
\rho_k(\xi)=\min\{|\xi|^{2k},1\}.
\]
These weights define the Kalman-adapted observation space
\(\mathcal Y_K^-\), and the dual controllability space is denoted by
\(\mathcal Y_K^+\). Thus the functional scale is not imposed externally; it is
selected by the low-frequency behavior of the Gramian.

For localized observation, we assume that the observation set satisfies a
one-dimensional bounded-gap condition. More precisely, there exist
\(\ell_{\mathcal O}>0\) and \(G_{\mathcal O}<\infty\) such that
\[
\mathcal O=\bigcup_{n\in\mathbb Z}(a_n,b_n),
\quad b_n<a_{n+1},
\]
and
\begin{equation}
\tag{BG}
\label{eq:gcc}
\ell_{\mathcal O}:=\inf_{n\in\mathbb Z}(b_n-a_n)>0,
\quad
G_{\mathcal O}:=\sup_{n\in\mathbb Z}(a_{n+1}-b_n)<\infty .
\end{equation}
We recall that a measurable set \(E\subset\mathbb R\) is said to be
\((\gamma,L)\)-thick if
\[
|E\cap (x,x+L)|\geq \gamma L,
\quad \forall x\in\mathbb R,
\]
for some constants \(L>0\) and \(\gamma\in(0,1]\). In the present
one-dimensional setting, the bounded-gap condition implies the usual
thickness condition from the Logvinenko--Sereda theory; see
\cite{LogvinenkoSereda1974,Kovrijkine2001}. Indeed, every interval of length
\[
L_{\mathcal O}:=\ell_{\mathcal O}+G_{\mathcal O}
\]
meets \(\mathcal O\) in a set of measure at least \(\ell_{\mathcal O}\).
Equivalently, \(\mathcal O\) is
\((\gamma_{\mathcal O},L_{\mathcal O})\)-thick, with
\[
\gamma_{\mathcal O}
=
\frac{\ell_{\mathcal O}}{\ell_{\mathcal O}+G_{\mathcal O}} .
\]
The two parameters play different roles. The lower bound
\(\ell_{\mathcal O}\) enters the spectral constants through the thickness
estimate, whereas the maximal gap \(G_{\mathcal O}\) determines the geometric
time threshold. If \(0\notin\sigma(A)\), set
\[
\lambda_*:=\min_{\lambda\in\sigma(A)}|\lambda|>0,
\quad
T_{\rm geom}:=\frac{G_{\mathcal O}}{\lambda_*}.
\]
This is the maximal time needed by the slowest characteristic family to cross
an unobserved gap. The condition \(T>T_{\rm geom}\) is therefore the natural
geometric threshold for uniform localized observability.

The main contributions are the following.
\begin{itemize}
\item We identify the Kalman-adapted Fourier scale selected by the finite-time
observation Gramian. This gives whole-line observability for every \(T>0\),
with the sharp small-time rate dictated by the deepest Kalman layer.

\item For diagonalizable \(A\) with real nonzero characteristic speeds, we prove
localized observability on sets satisfying \eqref{eq:gcc}, for every
\(T>T_{\rm geom}\). By HUM duality, this yields exact controllability in
\(\mathcal Y_K^+\), with quantitative bounds on the control cost.

\item We prove that the weights \(|\xi|^{2k}\) are optimal within the natural
class of block-diagonal Fourier multiplier norms adapted to the Kalman
filtration.

\item In the whole-line setting, we also treat matrices with real spectrum and
possible Jordan blocks. The corresponding Jordan-adapted scale combines the
low-frequency Kalman weights with the polynomial high-frequency weights
generated by the Jordan structure.
\end{itemize}

The word ``sharp'' is used in three different senses. First, the
low-frequency Kalman weights are optimal, both for whole-line and
localized observation, within the adapted block-diagonal Fourier
multiplier scale. Second, in the whole-line case, the time dependence
of the observability constant is optimal and is of order
\[
\max\{T^{-1},T^{-(2K+1)}\}.
\]
Third, in the localized problem, the geometric threshold is sharp in the
sense that uniform localized observability can fail for
\(T<T_{\mathrm{geom}}\); see
Remark~\ref{rem:geometric-threshold}. We do not claim that the
polynomial blow-up exponent of the localized cost as
\(T\downarrow T_{\mathrm{geom}}\) is optimal.

\subsection{Related work}
\label{subsec:related_work}
First-order hyperbolic and transport systems arise in wave propagation,
balance laws, networked systems, and distributed-parameter or
port-Hamiltonian models with partial observation or actuation
\cite{Lax2006,BastinCoron2016,JacobZwart2012,AugnerJacob2014};
see also \cite{XieDubljevic2021,XieEtAl2023} for recent transport examples.

The controllability and observability of hyperbolic systems are classical
topics in control theory. They are closely related to the Hilbert Uniqueness
Method of Lions and to Russell's work on exact controllability
\cite{Lions1988,Russell1978}; see also \cite{TucsnakWeiss2009}
for general background on control theory.

For systems on bounded domains, boundary conditions and reflections often play
an essential role in observability and controllability arguments. In several
situations, compactness--uniqueness methods are also used to close
observability estimates; see, for instance,
\cite{Kreiss1970,LiRao2003,Li2008}. Boundary control problems may further
couple incoming and outgoing characteristic modes through the boundary
operator. These features are specific to bounded-domain problems and are absent
on the real line.

Localized observation on unbounded sets is often based on
Logvinenko--Sereda spectral inequalities on thick sets, in the spirit of the
Lebeau--Robbiano method
\cite{LogvinenkoSereda1974,Kovrijkine2001,SEDP,LucMiller2010,Apraiz2014}.
In the present conservative transport setting, this spectral ingredient has to
be combined with propagation along characteristics and with the algebraic
structure of the observation.

Kalman-type conditions also play an important role in degenerate parabolic,
hypoelliptic, and parabolic--transport controllability problems
\cite{BEP,BeauchardKoenigLebalch2020,KoenigLissy23}. In particular, the works
of Beauchard--Koenig--Le Balc'h and Koenig--Lissy are close in spirit to the
present paper, since Kalman-type algebraic conditions are used there to
compensate for underactuation in coupled parabolic--transport systems.

The setting and the mechanism behind the loss are, however, different. In
Koenig--Lissy, the problem is posed on the one-dimensional torus. The Fourier
spectrum is therefore discrete, and the controllability criterion is formulated
mode by mode through a spectral Kalman rank condition. Moreover, the positive
controllability results are obtained for sufficiently regular Sobolev initial
data, and the authors also show that this regularity requirement cannot in
general be removed.

By contrast, the present paper deals with a purely first-order, undamped
transport system on the real line. There is no parabolic component and hence no
parabolic smoothing. The main obstruction is instead a low-frequency one: even
for global observation, the finite-time observation Gramian degenerates as the
continuous Fourier frequency tends to zero. Thus the Kalman structure does not
only appear as a rank condition on individual modes; it determines the precise
low-frequency weights and the Kalman-adapted observation space in which uniform
observability holds.

There is also a structural analogy with the Shizuta--Kawashima condition for
partially dissipative hyperbolic systems
\cite{ShizutaKawashima1985,XuKawashima2015,CrinBaratShouZuazua2025,
CRINBARAT2025103757}. In that theory, algebraic couplings transfer
dissipation from damped to undamped components. Here the transfer concerns
observability rather than decay. The analogy is therefore algebraic: in our
case, the coupling does not create dissipation, but determines how the
rank-deficient observation sees the different Kalman layers.

Finally, some coupled systems exploit an algebraic cascade by observing or
recovering higher-order quantities generated from the measured component, such
as \(B^\top(A^\top)^j q\); see, for instance, \cite{LissyZuazua2019}. By
contrast, the observation considered here is not augmented: the measured
quantity remains \(B^\top q\).

The present paper is concerned with this conservative real-line situation,
where boundary return, smoothing, dissipation, and augmented observations are
all absent. This separates the problem from the mechanisms above and leaves
the original rank-deficient observation as the only source of information.

\subsection{Organization of the paper}

Section~\ref{sec:main-results} states the main results and introduces the
Kalman-adapted spaces. Section~\ref{sec:whole-line} proves the whole-line
observability estimate through the Fourier--Kalman analysis of the finite-time
Gramian. Section~\ref{sec:localized-observability} proves localized observability on observation sets satisfying the bounded-gap condition \textup{(BG)}. Section~\ref{opti} establishes the optimality of
the Kalman-adapted weights. Section~\ref{jordann} treats the whole-line case
for matrices with real spectrum and possible Jordan blocks. Finally,
Section~\ref{sec:conclusion} discusses some perspectives, including the
localized Jordan case. The appendices collect auxiliary material on duality, stability of the
adapted spaces, admissibility estimates, and characteristic recovery under the bounded-gap condition \textup{(BG)}.

\section{Main results}\label{sec:main-results}

Throughout Sections 2--5, we assume that $A\in\mathbb R^{N\times N}$ is
diagonalizable over $\mathbb R$. The real-spectrum non-diagonalizable
whole-line case is treated separately in Section~6.

This section states the main observability and controllability results. We
first introduce the Kalman-adapted spaces dictated by the low-frequency
behavior of the observation Gramian, and then state the whole-line and
localized observability estimates, followed by their controllability
consequences.

We consider the controlled system \eqref{eq:controlled-system} and its adjoint
\eqref{adjoint-system}. The observation set is either \(\mathcal O=\mathbb R\) or a proper subset
of \(\mathbb R\) satisfying the bounded-gap condition \textup{(BG)}. Orthogonality and projectors are always
understood with respect to the Euclidean inner product on $\mathbb R^N$.

We assume that the pair $(A,B)$ satisfies the Kalman rank condition. Namely,
$K$ denotes the smallest integer such that
\begin{equation}\label{Kalman-condition}
\operatorname{rank}(B,AB,\dots,A^K B)=N.
\end{equation}
\subsection{Kalman layers and the observation space}
\label{subsec:kalman-layers-observation-space}
We use the Kalman filtration 
\[
\mathcal K_k:=\operatorname{Ran}(B,AB,\dots,A^kB),
\quad
\mathcal K_{-1}:=\{0\},
\]
and the associated orthogonal layers
\[
\mathcal V_k:=\mathcal K_k\cap\mathcal K_{k-1}^{\perp},
\quad k=0,\dots,K.
\]
Since \(\mathcal K_K=\mathbb R^N\), one has
$\mathbb R^N=\bigoplus_{k=0}^K\mathcal V_k.$ These elementary Kalman facts are proved in Lemma~\ref{lem:kalman} below.
Let \(\Pi_k\) denote the orthogonal projector onto \(\mathcal V_k\). The projectors \(\Pi_k\) are associated with the Kalman filtration. They are not
spectral projectors of \(A^\top\), and in general they do not commute with
\(A^\top\). Then 
\[
I_N=\sum_{k=0}^K\Pi_k,
\quad
B^\top(A^\top)^\ell\Pi_k=0
\quad\text{for } \ell<k.
\]
Moreover, \(B^\top(A^\top)^k\) is injective on \(\mathcal V_k\). Thus
\(\mathcal V_k\) is the layer first detected at algebraic depth \(k\).
For \(k=0,\dots,K\), set
\[
\rho_k(\xi):=\min\{|\xi|^{2k},1\}.
\]
These weights encode the precise low-frequency loss detected by the finite-time Gramian. Namely, the weight on \(\mathcal V_k\) is not an external regularity
assumption, but the first non-vanishing order at which
\(B^\top e^{is\xi A^\top}\) detects that layer near \(\xi=0\).

Let \(\mathcal S(\mathbb R;\mathbb R^N)\) denote the Schwartz space of
smooth rapidly decaying functions from \(\mathbb R\) to \(\mathbb R^N\).
All projectors \(\Pi_k\) are extended complex-linearly to \(\mathbb C^N\).
For \(q\in \mathcal S(\mathbb R;\mathbb R^N)\), we set
\begin{equation}\label{def:YminusK}
\|q\|_{\mathcal Y_K^{-}}^2
:=
\sum_{k=0}^K
\int_{\mathbb R}
\rho_k(\xi)\,|\Pi_k\widehat q(\xi)|^2\,d\xi .
\end{equation}
We define \(\mathcal Y_K^-\) as the Hilbert completion of
\(\mathcal S(\mathbb R;\mathbb R^N)\) for the norm
\eqref{def:YminusK}. Elements of \(\mathcal Y_K^-\) are therefore
understood as equivalence classes in this completion. Equivalently, \(\mathcal Y_K^-\) can be identified on the Fourier side
with the closure, in the norm \eqref{def:YminusK}, of the Fourier
transforms of functions in \(\mathcal S(\mathbb R;\mathbb R^N)\). In particular, \(\mathcal Y_K^-\) should be viewed
as a Kalman-weighted observation space, and not as the standard
\(L^2(\mathbb R)^N\) space.

With this identification, if \(P_{\leq 1}\) and \(P_{>1}\) denote smooth
Fourier cutoffs adapted respectively to low and high frequencies, and if
\(|D_x|\) denotes the Fourier multiplier with symbol \(|\xi|\), then
\[
\|q\|_{\mathcal Y_K^-}^2
\simeq
\sum_{k=0}^K
\left(
\|P_{\leq 1}|D_x|^k\Pi_k q\|_{L^2(\mathbb R;\mathbb R^N)}^2
+
\|P_{>1}\Pi_k q\|_{L^2(\mathbb R;\mathbb R^N)}^2
\right).
\]
Thus the component \(\Pi_k q\) is measured through its \(k\)-th low-frequency
derivative, while its high-frequency part is measured directly in \(L^2\).
On deeper Kalman layers, the norm involves higher powers of the low-frequency
multiplier, while no additional high-frequency regularity is imposed.

The corresponding dual space, with respect to the \(L^2\)-Fourier pairing, is
denoted by \(\mathcal Y_K^+\). Equivalently,
\begin{equation}\label{def:YplusK}
\|y\|_{\mathcal Y_K^+}^2
:=
\sum_{k=0}^K
\int_{\mathbb R}
\rho_k(\xi)^{-1}|\Pi_k\widehat y(\xi)|^2\,d\xi<\infty .
\end{equation}
Here \(\rho_k(\xi)^{-1}\) is understood for \(\xi\neq0\), with an arbitrary
value assigned at \(\xi=0\), which does not affect the integral.
For \(k\geq1\), this multiplier is singular as \(\xi\to0\).

Since Fourier transforms are identified up to sets of measure zero, the value
of the multiplier at the single point \(\xi=0\) is irrelevant; what matters is
its behavior as \(\xi\to0\). In physical variables,
\[
\|y\|_{\mathcal Y_K^+}^2
\simeq
\sum_{k=0}^K
\left(
\|P_{\leq 1}|D_x|^{-k}\Pi_k y\|_{L^2(\mathbb R)}^2
+
\|P_{>1}\Pi_k y\|_{L^2(\mathbb R)}^2
\right).
\]
Thus, in the dual space, the component \(\Pi_k y\) is admissible precisely when
its low-frequency primitive of order \(k\) belongs to \(L^2\), while its
high-frequency part remains in the standard \(L^2\) scale.

We shall also use the Fourier-side observability Gramian
\[
\mathcal G_T(\xi)
:=
\int_0^T
e^{-i s\xi A}
BB^\top
e^{i s\xi A^\top}\,ds.
\]

The results below are stated as observability estimates for the adjoint system.
Unless otherwise stated, adjoint solutions with terminal data in
\(\mathcal Y_K^-\) are understood by completion from smooth terminal data.
Equivalently, the homogeneous adjoint flow is first defined on
\(\mathcal S(\mathbb R;\mathbb R^N)\) and then extended continuously to
\(\mathcal Y_K^-\). The corresponding controlled solutions are understood in
the transposition sense associated with the duality between
\(\mathcal Y_K^-\) and \(\mathcal Y_K^+\).
Figure~\ref{fig:kalman} summarizes this layerwise Kalman frequency structure.
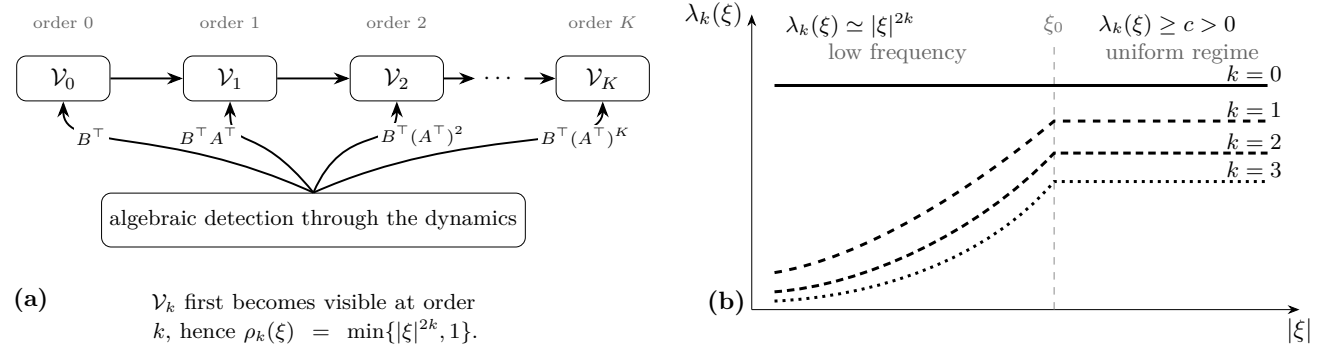
\begin{figure}[H]
\centering

\makebox[\textwidth][c]{%
\resizebox{0.49\textwidth}{!}{%
\begin{tikzpicture}[
  >=Stealth,
  node distance=1.05cm,
  every node/.style={transform shape},
  layer/.style={
    draw,
    rounded corners,
    minimum width=1.35cm,
    minimum height=0.65cm,
    align=center
  },
  obs/.style={
    draw,
    rounded corners,
    minimum width=5.4cm,
    minimum height=0.78cm,
    align=center
  },
  arr/.style={->,thick},
  lab/.style={font=\scriptsize,fill=white,inner sep=1pt},
  soft/.style={black!60}
]

\node[layer] (v0) at (0,0) {$\mathcal V_0$};
\node[layer, right=of v0] (v1) {$\mathcal V_1$};
\node[layer, right=of v1] (v2) {$\mathcal V_2$};
\node[right=0.45cm of v2] (dots) {$\cdots$};
\node[layer, right=0.45cm of dots] (vK) {$\mathcal V_K$};

\draw[arr] (v0) -- (v1);
\draw[arr] (v1) -- (v2);
\draw[arr] (v2) -- (dots);
\draw[arr] (dots) -- (vK);

\node[soft,above=0.28cm of v0] {\scriptsize order $0$};
\node[soft,above=0.28cm of v1] {\scriptsize order $1$};
\node[soft,above=0.28cm of v2] {\scriptsize order $2$};
\node[soft,above=0.28cm of vK] {\scriptsize order $K$};

\node[obs] (obsbox) at ($(v1)!0.5!(v2)+(0,-2.05)$)
{\small algebraic detection through the dynamics};

\draw[arr] (obsbox.north) .. controls +(-2.2,1.0) and +(0,-0.75) ..
node[lab,pos=0.58,left=1pt] {$B^\top$} (v0.south);

\draw[arr] (obsbox.north) .. controls +(-0.85,1.05) and +(0,-0.75) ..
node[lab,pos=0.58,left=1pt] {$B^\top A^\top$} (v1.south);

\draw[arr] (obsbox.north) .. controls +(0.45,1.05) and +(0,-0.75) ..
node[lab,pos=0.58,right=1pt] {$B^\top(A^\top)^2$} (v2.south);

\draw[arr] (obsbox.north) .. controls +(2.25,1.0) and +(0,-0.75) ..
node[lab,pos=0.56,right=1pt] {$B^\top(A^\top)^K$} (vK.south);

\node[below=0.55cm of obsbox,align=center,text width=7.1cm]
{\small
$\mathcal V_k$ first becomes visible at order $k$,
hence $\rho_k(\xi)=\min\{|\xi|^{2k},1\}$.
};

\node[anchor=north west,font=\bfseries] at (-0.85,-2.90) {(a)};

\end{tikzpicture}%
}%
\hspace{0.02\textwidth}%
\resizebox{0.49\textwidth}{!}{%
\begin{tikzpicture}[>=Stealth]

\draw[->] (0,0) -- (7.7,0) node[below] {$|\xi|$};
\draw[->] (0,0) -- (0,4.15) node[left] {$\lambda_k(\xi)$};

\def\xth{4.25}

\draw[dashed,black!45] (\xth,0) -- (\xth,3.75);
\node[black!55,anchor=south] at (\xth,3.75) {\small $\xi_0$};

\node[black!60] at (2.05,3.6) {\small low frequency};
\node[black!60] at (6.05,3.6) {\small uniform regime};

\draw[very thick] (0.32,3.15) -- (7.25,3.15);
\node[anchor=west] at (6.55,3.32) {\small $k=0$};

\draw[very thick,dashed]
(0.32,0.52) .. controls (1.55,0.72) and (3.05,1.65) .. (\xth,2.65)
-- (7.25,2.65);
\node[anchor=west] at (6.55,2.82) {\small $k=1$};

\draw[very thick,densely dashed]
(0.32,0.25) .. controls (1.55,0.34) and (3.15,1.05) .. (\xth,2.20)
-- (7.25,2.20);
\node[anchor=west] at (6.55,2.37) {\small $k=2$};

\draw[very thick,dotted]
(0.32,0.12) .. controls (1.70,0.18) and (3.30,0.68) .. (\xth,1.80)
-- (7.25,1.80);
\node[anchor=west] at (6.55,1.97) {\small $k=3$};

\node[anchor=west] at (0.35,4.00)
{\small $\lambda_k(\xi)\simeq |\xi|^{2k}$};

\node[anchor=west] at (4.75,4.00)
{\small $\lambda_k(\xi)\ge c>0$};

\node[anchor=north west,font=\bfseries] at (-0.75,0.40) {(b)};

\end{tikzpicture}%
}%
}

\caption{\textbf{ Kalman layers: low-frequency weights.}
The Kalman layer $\mathcal V_k$ becomes observable at algebraic order $k$,
which yields the low-frequency Gramian scale $|\xi|^{2k}$ and the adapted
weight $\rho_k(\xi)=\min\{|\xi|^{2k},1\}$.
Here $\lambda_k(\xi):=
\lambda_{\min}\!\left(\Pi_k\mathcal G_T(\xi)\Pi_k|_{\mathcal V_k}\right)$
denotes the smallest Gramian eigenvalue on $\mathcal V_k$.
}
\label{fig:kalman}
\end{figure}
\subsection{Whole-line and localized observability}

We now state the two main observability estimates. We use the notation introduced above: \(\mathcal O\) denotes the observation set and \(T_{\mathrm{geom}}\) the associated geometric time. In the localized case,
we also set \(p:=\#\sigma(A)\), the number of distinct eigenvalues of \(A\).

We first state the whole-line estimate, which is purely Fourier--Kalman in nature. 
We then turn to the localized estimate, where the adapted scale is combined 
with characteristic propagation and thick-set spectral inequalities.
\begin{theorem}\label{thm:global}
Assume that \(A\) is diagonalizable over \(\mathbb R\) and that \((A,B)\)
satisfies the Kalman rank condition~\eqref{Kalman-condition}. Let
\(q_T\in\mathcal Y_K^{-}\) and let \(q\) be the corresponding solution of the
adjoint system~\eqref{adjoint-system}.

If \(\mathcal O=\mathbb R\), then for every \(T>0\) there exists a constant
\(C_*>0\), depending only on \(A\), \(B\), and \(K\), such that
\begin{equation}\label{eq:obs-global}
\|q_T\|_{\mathcal Y_K^{-}}^2
\leq
C_*
\max\left\{\frac1T,\frac1{T^{2K+1}}\right\}
\int_0^T\int_{\mathbb R}
|B^\top q(t,x)|^2\,dx\,dt.
\end{equation}
Moreover, the time dependence
 $\max\left\{T^{-1},T^{-(2K+1)}\right\}$
is optimal.
\end{theorem}
The next result shows that the same Kalman-adapted scale remains observable from localized measurements, provided the observation set satisfies the bounded-gap condition \textup{(BG)} and the observation
time is larger than the geometric threshold.
\begin{theorem}\label{thm:local}
Assume that \(A\) is diagonalizable over \(\mathbb R\), that \((A,B)\)
satisfies the Kalman rank condition~\eqref{Kalman-condition}, and that
\(0\notin\sigma(A)\). Assume also that \(\mathcal O\) satisfies the bounded-gap condition~\eqref{eq:gcc}. Let $p:=\#\sigma(A).$
Then, for every \(T_1>T_{\mathrm{geom}}\), there exists a constant
\(C_{\mathcal O,T_1}>0\), depending only on \(\mathcal O\), \(A\), \(B\),
\(K\), and \(T_1\), but independent of \(T\), such that, for every
\(T\in(T_{\mathrm{geom}},T_1]\), every \(q_T\in\mathcal Y_K^{-}\), and every
corresponding solution \(q\) of the adjoint system~\eqref{adjoint-system}, one has
\begin{equation}\label{eq:local}
\|q_T\|_{\mathcal Y_K^{-}}^2
\leq
\frac{C_{\mathcal O,T_1}}
{(T-T_{\mathrm{geom}})^{2p-1}}
\int_0^T\int_{\mathcal O}
|B^\top q(t,x)|^2\,dx\,dt.
\end{equation}
\end{theorem}
\begin{remark}
\label{rem:geometric-threshold}
The exponent \(2p-1\) in \eqref{eq:local} is not claimed to be optimal. It is the loss produced by the finite-order differential separation argument. Thus, for every \(T_1>T_{\mathrm{geom}}\), the estimate provides only the locally uniform upper bound
\[
C_{\mathrm{obs}}(T)
\lesssim_{\,\mathcal O,A,B,K,T_1}
(T-T_{\mathrm{geom}})^{-(2p-1)},
\quad
T\in(T_{\mathrm{geom}},T_1],
\]
and this should not be interpreted as a sharp asymptotic law for the observability cost as \(T\downarrow T_{\mathrm{geom}}\).

By contrast, the localized result has two sharp features: the
Kalman-adapted scale \(\mathcal Y_K^{-}\) and the geometric travel-time threshold. The condition \(T>T_{\mathrm{geom}}\) is the natural characteristic travel-time condition. Below this threshold, the spatial trace of a characteristic segment may remain entirely inside an unobserved gap. High-frequency wave packets may therefore avoid the observation region, so uniform localized observability on the full space \(\mathcal Y_K^{-}\) can fail. This is the one-dimensional transport analogue of the classical geometric-optics obstruction; see \cite{BardosLebeauRauch1992}. The assumption
\(0\notin\sigma(A)\) excludes stationary characteristic modes trapped in such gaps. No claim is made here concerning the endpoint \(T=T_{\mathrm{geom}}\).

For spectrally localized terminal data, this high-frequency obstruction disappears. More precisely, if
\[
\operatorname{supp}\widehat q_T\subset[-R,R],
\]
then Proposition~\ref{prop:bounded-frequency-localization} yields
localized observability for every \(T>0\), with a constant depending on
\(T\), \(R\), and the parameters of \(\mathcal O\). This does not alter
the uniform geometric threshold, since the obstruction below
\(T_{\mathrm{geom}}\) is a high-frequency one.
\end{remark}

\begin{corollary}\label{equivalent-norm}
Assume the hypotheses of Theorem~\ref{thm:local}. Let
\(T_1>T_{\mathrm{geom}}\) and let \(T\in(T_{\mathrm{geom}},T_1]\). Then there exist constants
\(\mathcal C_T,C_{\mathcal O,T_1}>0\), with \(C_{\mathcal O,T_1}\) independent of \(T\), such that, for every

\(q_T\in\mathcal Y_K^{-}\),
\[
\frac{(T-T_{\mathrm{geom}})^{2p-1}}{C_{\mathcal O,T_1}}
\|q_T\|_{\mathcal Y_K^{-}}^2
\leq
\int_0^T\int_{\mathcal O}|B^\top q(t,x)|^2\,dx\,dt
\leq
\mathcal C_{T}\|q_T\|_{\mathcal Y_K^{-}}^2 .
\]
Thus, for every \(T>T_{\mathrm{geom}}\), the localized observation functional
induces a norm equivalent to the Kalman-adapted norm on \(\mathcal Y_K^{-}\).
\end{corollary}

\begin{proof}
The left-hand inequality is the localized observability estimate
\eqref{eq:local}. The right-hand inequality is the admissibility estimate
proved in Appendix~\ref{app:norm-equivalence}.
\end{proof}

\subsection{Controllability consequences}

By the Hilbert Uniqueness Method, the observability estimates above yield the
following exact controllability results.

\begin{corollary}\label{cor:control}
Assume that $A$ is diagonalizable over $\mathbb R$ and that $(A,B)$ satisfies
the Kalman rank condition. Let $\mathcal Y_K^+$ denote the dual space
associated with $\mathcal Y_K^-$ through the $L^2$-Fourier pairing. Then the
controlled system \eqref{eq:controlled-system} is exactly controllable in
$\mathcal Y_K^+$ in the following sense.

\smallskip
\noindent\textbf{(i) Global case.}
Let $\mathcal O=\mathbb R$. For every $T>0$ and every
$y_0,y_d\in\mathcal Y_K^{+}$, there exists
$u\in L^2((0,T)\times\mathbb R;\mathbb R^m)$ such that the corresponding solution satisfies
$y(T,\cdot)=y_d.$
Moreover,
\[
\|u\|_{L^2((0,T)\times\mathbb R;\mathbb R^m)}^2
\leq
\widetilde C_* \max\left\{T^{-1},T^{-(2K+1)}\right\}
\|y_d-S(T)y_0\|_{\mathcal Y_K^{+}}^2,
\]
where $S(T)$ denotes the free evolution.

\smallskip
\noindent\textbf{(ii) Localized case.}
Assume in addition that $0\notin\sigma(A)$ and that $\mathcal O$ satisfies
the bounded-gap condition~\eqref{eq:gcc}. Let \(T_1>T_{\mathrm{geom}}\). Then
there exists a constant \(\widetilde C_{\mathcal O,T_1}>0\) such that, for every
\(T\in(T_{\mathrm{geom}},T_1]\) and every
$y_0,y_d\in\mathcal Y_K^{+}$, there exists
$u\in L^2((0,T)\times\mathcal O;\mathbb R^m)$ such that
$y(T,\cdot)=y_d.$
Moreover,
\[
\|u\|_{L^2((0,T)\times\mathcal O;\mathbb R^m)}^2
\leq
\frac{\widetilde C_{\mathcal O,T_1}}
{(T-T_{\mathrm{geom}})^{2p-1}}
\|y_d-S(T)y_0\|_{\mathcal Y_K^{+}}^2.
\]
\end{corollary}
\begin{remark}\label{rem:control-space}
In Corollary~\ref{cor:control}, the controllability space is
$\mathcal Y_K^+$, not $L^2(\mathbb R)^N$. This is the dual manifestation of
the low-frequency degeneracy of the observation Gramian. The value of the
weight at $\xi=0$ is irrelevant; only its behavior near zero matters. The
optimality of the layerwise weights is established in Section~\ref{opti}.
\end{remark}
\subsection{An Euler-type illustration of a depth-two Kalman filtration}
\label{subsec:3d-euler-example}

We illustrate the Kalman-adapted scale on a simple Euler-type linearized
system. Such matrices arise, for instance, from the primitive-variable form of
the one-dimensional compressible Euler equations after linearization around a
constant state; see \cite{LeVeque2002,Dafermos2016}. Consider
\[
A=
\begin{pmatrix}
2&1&0\\
1&2&1\\
0&0&2
\end{pmatrix},
\quad
\sigma(A)=\{1,2,3\},
\]
and the adjoint system
\[
\partial_t q+A^\top\partial_x q=0,\quad q=(q_1,q_2,q_3)^\top .
\]
We observe only the third adjoint component,
\[
B=e_3,\qquad B^\top q=q_3.
\]
Since
\[
B=e_3,\quad
AB=(0,1,2)^\top,\quad
A^2B=(1,4,4)^\top
\]
form a basis of \(\mathbb R^3\), the Kalman condition holds with index \(K=2\).
The associated Kalman filtration is
\[
\mathcal K_0=\operatorname{span}\{e_3\},\quad
\mathcal K_1=\operatorname{span}\{e_2,e_3\},\quad
\mathcal K_2=\mathbb R^3,
\]
and hence
\[
\mathcal V_0=\operatorname{span}\{e_3\},\quad
\mathcal V_1=\operatorname{span}\{e_2\},\quad
\mathcal V_2=\operatorname{span}\{e_1\}.
\]
A direct computation gives detection orders \(0,1,2\) on
\(\mathcal V_0,\mathcal V_1,\mathcal V_2\), respectively. Therefore the
low-frequency Gramian scale is
\[
\mathcal V_0:\ 1,\quad
\mathcal V_1:\ |\xi|^2,\quad
\mathcal V_2:\ |\xi|^4 .
\]
Equivalently, the adapted norm is
\[
\|q\|_{Y_2^-}^2
=
\int_{\mathbb R}|\widehat q_3(\xi)|^2\,d\xi
+
\int_{\mathbb R}\min\{|\xi|^2,1\}|\widehat q_2(\xi)|^2\,d\xi
+
\int_{\mathbb R}\min\{|\xi|^4,1\}|\widehat q_1(\xi)|^2\,d\xi .
\]
Thus the observed adjoint direction is \(e_3\), while \(e_2\) and \(e_1\) are
recovered after one and two algebraic couplings, respectively.

\section{Whole-line observability: Fourier reduction and frequency regimes}
\label{sec:whole-line}

We prove the whole-line observability estimate by working frequency by
frequency. The proof is divided into a low-frequency estimate, based on the
Kalman layer decomposition, and a nonzero-frequency estimate, obtained by
compactness on bounded annuli and spectral separation at high frequencies.

\subsection{Low-frequency Gramian asymptotics and Kalman layer structure}
\label{lowfrequency}
\paragraph{\textbf{Step 1: Fourier reduction and observability Gramian.}}
Taking the Fourier transform in $x$ of the adjoint system gives, for each
$\xi\in\mathbb R$,
\[
\partial_t \widehat q+i\xi A^\top\widehat q=0,
\quad
\widehat q(T,\xi)=\widehat q_T(\xi).
\quad
\text{Hence}
\quad
\widehat q(t,\xi)
=
e^{i\xi(T-t)A^\top}\widehat q_T(\xi).
\]
Here $\widehat q(t,\xi)\in\mathbb C^N$, and all real matrices are extended to
$\mathbb C^N$ by complex linearity.

With the change of variable $s=T-t$, define the Hermitian observability
Gramian
\begin{equation}\label{eq:gramian}
\mathcal G_T(\xi)
:=
\int_0^T
e^{-is\xi A}
BB^\top
e^{is\xi A^\top}
\,ds .
\end{equation}
Equivalently, for every $z\in\mathbb C^N$,
\begin{equation}\label{eq:gramiann}
z^*\mathcal G_T(\xi)z
=
\int_0^T
\|B^\top e^{is\xi A^\top}z\|_{\mathbb C^m}^2\,ds
\geq 0 .
\end{equation}
Thus $\mathcal G_T(\xi)$ is Hermitian positive semidefinite. Moreover, the map
$\xi\mapsto \mathcal G_T(\xi)$ is continuous.

\medskip
\paragraph{\textbf{Step 2: A finite-dimensional coercivity.}}
We recall the following standard finite-dimensional fact.
Let $p_0,\dots,p_K\in L^2(0,1)$ be linearly independent and define the Gram matrix
\[
M_{k\ell}:=\int_0^1 \overline{p_k(\tau)}\,p_\ell(\tau)\,d\tau,
\quad 0\le k,\ell\le K .
\]
Then $M$ is Hermitian positive definite.
Fix the Hermitian inner product on $\C^m$ as $\langle v,w\rangle:=v^*w$ and set
\[
U = (\tilde u_0,\tilde u_1,\dots,\tilde u_K)^\top \in \C^{m(K+1)}.
\]
Then one has the exact identity
\[
\int_0^1\Big\|\sum_{k=0}^K p_k(\tau)\tilde u_k\Big\|_{\C^m}^2\,d\tau
= U^*(M\otimes I_m)U,
\]
and therefore
\begin{equation}\label{gramm}
\int_0^1\Big\|\sum_{k=0}^K p_k(\tau)\tilde u_k\Big\|_{\C^m}^2\,d\tau
\ge \lambda_{\min}(M)\sum_{k=0}^K\|\tilde u_k\|_{\C^m}^2 .
\end{equation}
We shall apply this estimate to the leading Kalman jets.
\paragraph{\textbf{Step 3: Kalman decomposition and coercivity on each layer.}}
\begin{lemma}\label{lem:kalman}
Let \(\mathcal V_k\) be the Kalman layers defined in
Section~\ref{subsec:kalman-layers-observation-space}. If the Kalman rank condition holds with index \(K\), then
\[
\mathbb R^N=\bigoplus_{k=0}^K\mathcal V_k,
\quad
B^\top(A^\top)^k\big|_{\mathcal V_k}\ \text{is injective for } k=0,\dots,K.
\] Consequently, there exists \(c_k>0\) such that
\begin{equation}\label{eq:ck-single}
\|B^\top(A^\top)^k v\|_{\mathbb R^m}
\ge c_k\|v\|_{\mathbb R^N},
\qquad v\in\mathcal V_k .
\end{equation}
\end{lemma}

\begin{proof}
Fix $k\ge 0$ and let $v\in \mathcal V_k$ satisfy
$B^\top (A^\top)^k v=0.$
We first recall that for every $\ell\ge 0$,
\[
\bigl(\operatorname{Ran}(A^\ell B)\bigr)^\perp
=
\ker\bigl((A^\ell B)^\top\bigr)
=
\ker\bigl(B^\top (A^\top)^\ell\bigr).
\]
Hence
\begin{equation}\label{eq:Sk-orth-single}
\mathcal K_k^\perp
=
\Big(\sum_{\ell=0}^k \operatorname{Ran}(A^\ell B)\Big)^\perp
=
\bigcap_{\ell=0}^k \ker\bigl(B^\top (A^\top)^\ell\bigr).
\end{equation}

Now, since $v\in \mathcal V_k=\mathcal K_k\cap \mathcal K_{k-1}^\perp$, we have
$v\in \mathcal K_k
\quad\text{and}\quad
v\in \mathcal K_{k-1}^\perp.$
Because
\[
\mathcal K_{k-1}=\sum_{\ell=0}^{k-1}\operatorname{Ran}(A^\ell B),
\quad 
\text{the relation} \quad v\in \mathcal K_{k-1}^\perp\quad  \text{implies}
\quad 
B^\top (A^\top)^\ell v=0,
\qquad \ell=0,\dots,k-1.
\]
Together with the assumption $B^\top (A^\top)^k v=0$, this yields
\[
v\in \bigcap_{\ell=0}^k \ker\bigl(B^\top (A^\top)^\ell\bigr)=\mathcal K_k^\perp
\]
by \eqref{eq:Sk-orth-single}. Since already $v\in \mathcal K_k$, we conclude that
$v\in \mathcal K_k\cap \mathcal K_k^\perp=\{0\},$
hence $v=0$. This proves that $B^\top (A^\top)^k:\mathcal V_k\to\R^m$ is injective.
Since $\mathcal V_k$ is finite-dimensional, injectivity implies the existence of
$c_k>0$ such that
\[
\|B^\top (A^\top)^k v\|_{\R^m}\ge c_k\|v\|_{\R^N},
\qquad \forall\,v\in\mathcal V_k.
\]
\end{proof}
\medskip
\paragraph{\textbf{Step 4: Low-frequency expansion and coercivity.}}
We use the complexified Kalman layers
\[
\mathcal V_k^{\mathbb C}
:=\mathcal V_k\otimes_{\mathbb R}\mathbb C
=\{v+iw:\ v,w\in\mathcal V_k\}\subset\mathbb C^N.
\]
\begin{lemma}\label{lem:lowfreq}
Assume that
\[
\mathbb R^N=\bigoplus_{k=0}^K\mathcal V_k
\]
is the orthogonal Kalman decomposition of Lemma~\ref{lem:kalman}. Define
\[
q^*\mathcal G_T(\xi)q
:=
\int_0^T
\|B^\top e^{is\xi A^\top}q\|_{\mathbb C^m}^2\,ds .
\]
Then there exist constants \(\eta_0>0\) and \(c>0\), depending only on
\((A,B,K)\), such that, for all \(T>0\), all \(T|\xi|\le \eta_0\), and all
\(q=\sum_{k=0}^K q_k\) with \(q_k\in\mathcal V_k^{\mathbb C}\),
\begin{equation}\label{eq:lowfreq-global}
q^*\mathcal G_T(\xi)q
\ge
c\sum_{k=0}^K
T^{2k+1}|\xi|^{2k}\|q_k\|_{\mathbb C^N}^2 .
\end{equation}
In particular,
\begin{equation}\label{eq:lowfreq-global-min}
q^*\mathcal G_T(\xi)q
\ge
c\min\{T,T^{2K+1}\}
\sum_{k=0}^K
|\xi|^{2k}\|q_k\|_{\mathbb C^N}^2 .
\end{equation}
\end{lemma}

\begin{proof}
Write \(q=\sum_{k=0}^K q_k\), with \(q_k\in\mathcal V_k^{\mathbb C}\). By the
triangular Kalman relations,
\[
B^\top(A^\top)^\ell q_k=0,\qquad \ell<k.
\]
Hence Taylor's formula gives, uniformly for \(0\le s\le T\) and
\(T|\xi|\le\eta_*\),
\[
B^\top e^{is\xi A^\top}q_k
=
(i\xi)^k\frac{s^k}{k!}B^\top(A^\top)^kq_k
+
r_{k,\xi}(s),
\qquad
\|r_{k,\xi}(s)\|
\le
C|\xi|^{k+1}s^{k+1}\|q_k\|.
\]
Set
\[
u_k:=B^\top(A^\top)^kq_k,\qquad
S_\xi(s):=\sum_{k=0}^K(i\xi)^k\frac{s^k}{k!}u_k,
\qquad
R_\xi(s):=\sum_{k=0}^K r_{k,\xi}(s).
\]
Then
\[
B^\top e^{is\xi A^\top}q=S_\xi(s)+R_\xi(s).
\]

Let \(s=T\tau\). Since \(1,\tau,\dots,\tau^K\) are linearly independent in
\(L^2(0,1)\), the finite-dimensional Gram estimate and Lemma~\ref{lem:kalman}
give
\[
\int_0^T\|S_\xi(s)\|^2\,ds
\ge
c_0\sum_{k=0}^K
T^{2k+1}|\xi|^{2k}\|q_k\|^2 ,
\]
where \(c_0>0\) depends only on \((A,B,K)\).

On the other hand, using the remainder estimate and the finiteness of the sum over
\(k\),
\[
\int_0^T\|R_\xi(s)\|^2\,ds
\le
C(T|\xi|)^2
\sum_{k=0}^K
T^{2k+1}|\xi|^{2k}\|q_k\|^2 .
\]
Therefore, from
\[
\|a+b\|^2\ge \frac12\|a\|^2-\|b\|^2,
\]
we obtain
\[
q^*\mathcal G_T(\xi)q
\ge
\left(\frac{c_0}{2}-C(T|\xi|)^2\right)
\sum_{k=0}^K
T^{2k+1}|\xi|^{2k}\|q_k\|^2 .
\]
Choosing \(\eta_0>0\) so small that
$C\eta_0^2\le \frac{c_0}{4},$
we get, for \(T|\xi|\le\eta_0\),
\[
q^*\mathcal G_T(\xi)q
\ge
\frac{c_0}{4}
\sum_{k=0}^K
T^{2k+1}|\xi|^{2k}\|q_k\|^2 .
\]
This proves \eqref{eq:lowfreq-global}. Finally, since
\[
T^{2k+1}\ge \min\{T,T^{2K+1}\},
\quad 0\le k\le K,
\]
we obtain \eqref{eq:lowfreq-global-min}.
\end{proof}
\begin{corollary}
\label{cor:layerwise-asymptotics}
Under the assumptions of Lemma~\ref{lem:lowfreq}, fix
\(k\in\{0,\dots,K\}\). Then there exist constants
\(c_k>0\), \(C_k>0\), and \(\eta_k>0\), depending only on
\((A,B,K)\), such that, for every \(T>0\), every
\(q\in\mathcal V_k^{\mathbb C}\), and every \(T|\xi|\le \eta_k\),
\begin{equation}\label{eq:layerwise-asymptotics}
c_k\,T^{2k+1}|\xi|^{2k}\|q\|_{\mathbb C^N}^2
\le
q^*\mathcal G_T(\xi)q
\le
C_k\,T^{2k+1}|\xi|^{2k}\|q\|_{\mathbb C^N}^2 .
\end{equation}
Equivalently, for each fixed \(T>0\),
\[
q^*\mathcal G_T(\xi)q
\asymp
T^{2k+1}|\xi|^{2k}\|q\|_{\mathbb C^N}^2,
\qquad
q\in\mathcal V_k^{\mathbb C},\quad |\xi|\to0 .
\]
The constants in this equivalence are uniform in the low-frequency regime
\(T|\xi|\le \eta_k\).

In operator form, if \(\Pi_k\) denotes the orthogonal projector onto
\(\mathcal V_k\), then, for each fixed \(T>0\),
\[
\lambda_{\min}\!\left(
\Pi_k \mathcal G_T(\xi)\Pi_k\big|_{\mathcal V_k^{\mathbb C}}
\right)
\asymp
T^{2k+1}|\xi|^{2k},
\quad |\xi|\to0 .
\]
Equivalently, the same estimate holds uniformly whenever
\(T|\xi|\le \eta_k\).
\end{corollary}
\begin{proof}
The triangular Kalman relations imply that the first nonzero term in the
expansion of \(B^\top e^{is\xi A^\top}q\) occurs at order \(k\). Hence
\[
B^\top e^{is\xi A^\top}q
=
(i\xi)^k\frac{s^k}{k!}B^\top(A^\top)^kq
+
O(|\xi|^{k+1}s^{k+1}\|q\|),
\]
uniformly for \(s\in[0,T]\). Since \(B^\top(A^\top)^k\) is injective on
\(\mathcal V_k\), integration in time gives the claimed two-sided estimate for
\(|\xi|\) small.
\end{proof}
\begin{remark}
The factor \(T^{2k+1}\) in Corollary~\ref{cor:layerwise-asymptotics}
is the natural small-time scaling of a Kalman chain of depth \(k\). Equivalently,
at fixed low frequency the inverse Gramian on \(\mathcal V_k\) has size
\(T^{-(2k+1)}\), so the corresponding control norm scales like
\(T^{-(k+1/2)}\). For the deepest layer \(k=K\), this agrees with the
classical finite-dimensional small-time blow-up rate for controllable linear
systems. The constants obtained in Lemma~\ref{lem:lowfreq} come from a
finite-dimensional moment estimate for the monomials
\(1,\tau,\dots,\tau^K\); they are positive for each fixed \(K\), but are not
intended to be quantitatively optimal.
\end{remark}
\subsection{Medium and high frequencies}
\label{subsec:medium-high}

It remains to treat \(|\xi|\geq\xi_0\), where \(\xi_0\in(0,1]\) is given by
Lemma~\ref{lem:lowfreq}.

\paragraph{\textbf{Step A: Uniform coercivity on bounded nonzero frequency bands.}}

For each fixed $\xi\neq0$, the Fourier modes satisfy
\[
\partial_t \widehat q+i\xi A^\top\widehat q=0,
\qquad
B^\top\widehat q \ \text{observed}.
\]
We first show that the Gramian $\mathcal G_T(\xi)$ is positive definite for every
$\xi\neq0$. Indeed, if
$z^*\mathcal G_T(\xi)z=0,$
then
\[
B^\top e^{is\xi A^\top}z=0,
\quad 0\leq s\leq T.
\]
Differentiating at $s=0$ gives
\[
B^\top(A^\top)^\ell z=0,
\quad \ell\geq0.
\]
By the Kalman rank condition, this implies $z=0$. Hence $\mathcal G_T(\xi)$ is positive
definite for every $\xi\neq0$.

Now fix $R>\xi_0$. Since $\xi\mapsto \mathcal G_T(\xi)$ is continuous as a Hermitian
matrix-valued function, the map $\xi\mapsto \lambda_{\min}(\mathcal G_T(\xi))$
is continuous on the compact set
$\{\xi\in\mathbb R:\xi_0\leq |\xi|\leq R\}.$
It is strictly positive there. Therefore, there exists
$c_{\mathrm{med}}>0$ such that
\begin{equation}\label{eq:medium-coercivity}
z^*\mathcal G_T(\xi)z
\geq
c_{\mathrm{med}}\|z\|_{\C^N}^2,
\quad
\xi_0\leq |\xi|\leq R.
\end{equation}

\paragraph{\textbf{Step B: High-frequency coercivity by spectral separation.}}

We consider \( |\xi|\to\infty \). Since \(A^\top\) is diagonalizable over
\(\mathbb R\), one has
\[
\mathbb R^N=\bigoplus_{\lambda\in\sigma(A)}E_\lambda,
\quad
E_\lambda:=\ker(A^\top-\lambda I).
\]
Let \(\mathcal P_\lambda\) denote the spectral projector onto \(E_\lambda\) associated
with this direct-sum decomposition. Thus
\[
\sum_{\lambda\in\sigma(A)}\mathcal P_\lambda=I,
\quad
A^\top \mathcal P_\lambda=\lambda \mathcal P_\lambda .
\]
Equivalently,
\[
\mathcal P_\lambda
=
\prod_{\substack{\mu\in\sigma(A)\\ \mu\neq\lambda}}
\frac{A^\top-\mu I}{\lambda-\mu}.
\]
Thus, for $z=\sum_\lambda \mathcal P_\lambda z$,
Writing $a_\lambda:=B^\top \mathcal P_\lambda z$, the Gramian gives
\[
\frac1T z^*\mathcal G_T(\xi)z
=
\sum_{\lambda}\|a_\lambda\|_{\C^m}^2
+
\sum_{\lambda\neq\nu}
\langle a_\lambda,a_\nu\rangle
\frac1T\int_0^T e^{is\xi(\nu-\lambda)}\,ds .
\]
For $\lambda\neq\nu$,
\[
\left|
\frac1T\int_0^T e^{is\xi(\nu-\lambda)}\,ds
\right|
\leq
\frac{2}{T|\xi|\,|\lambda-\nu|}.
\]
Since $\sigma(A)$ is finite, the spectral gaps are bounded from below; hence
there exists $C_A>0$ such that
\begin{equation}\label{eq:cross-high}
\left|
\sum_{\lambda\neq\nu}
\langle a_\lambda,a_\nu\rangle
\frac1T\int_0^T e^{is\xi(\nu-\lambda)}\,ds
\right|
\leq
\frac{C_A}{T|\xi|}
\sum_{\lambda}\|a_\lambda\|_{\C^m}^2 .
\end{equation}
It remains to control $\|z\|_{\C^N}^2$ by the diagonal part. The Kalman rank
condition implies that $B^\top$ is injective on each $E_\lambda$:
indeed, if $w\in E_\lambda$ and $B^\top w=0$, then
\[
B^\top(A^\top)^\ell w=\lambda^\ell B^\top w=0,
\quad \ell\geq0,
\]
hence $w=0$. Therefore
\[
f(z):=\sum_{\lambda\in\sigma(A)}
\|B^\top \mathcal P_\lambda z\|_{\C^m}^2
\]
is positive definite on $\mathbb C^N$. By compactness of the unit sphere,
there exists $\alpha>0$ such that
\begin{equation}\label{eq:spectral-coercivity}
\sum_{\lambda\in\sigma(A)}
\|B^\top \mathcal P_\lambda z\|_{\C^m}^2
\geq
\alpha\|z\|_{\C^N}^2,
\quad z\in\mathbb C^N .
\end{equation}
No orthogonality of the spectral projectors $P_\lambda$ is required.
Combining \eqref{eq:cross-high} and \eqref{eq:spectral-coercivity},
\[
\frac1T z^*\mathcal G_T(\xi)z
\geq
\left(1-\frac{C_A}{T|\xi|}\right)
\sum_{\lambda}
\|B^\top \mathcal P_\lambda z\|_{\C^m}^2 .
\]
Choosing $R_\infty>0$ so that $C_A/(T|\xi|)\leq1/2$ for
$|\xi|\geq R_\infty$, we obtain
\begin{equation}\label{eq:high-coercivity}
z^*\mathcal G_T(\xi)z
\geq
\frac{\alpha T}{2}\|z\|_{\C^N}^2,
\quad
|\xi|\geq R_\infty .
\end{equation}

\paragraph{\textbf{Conclusion and sharp time dependence in the whole-line estimate.}}
Set $\xi_T:=\frac{\eta_0}{T}.$
For \(|\xi|\leq\xi_T\), Lemma~\ref{lem:lowfreq} gives
\[
\widehat q_T(\xi)^*\mathcal G_T(\xi)\widehat q_T(\xi)
\geq
c_{\rm low}\min\{T,T^{2K+1}\}
\sum_{k=0}^K
\rho_k(\xi)\|\Pi_k\widehat q_T(\xi)\|_{\C^N}^2 .
\]
For \(|\xi|\geq\xi_T\), using the scaling identity
\[
\mathcal G_T(\xi)=T\mathcal G_1(T\xi),
\]
together with the medium- and high-frequency coercivity at time \(1\), there exists
\(c_{\rm mh}>0\), depending only on \((A,B,K)\), such that
\[
z^*\mathcal G_T(\xi)z
\geq
c_{\rm mh}T\|z\|_{\C^N}^2,
\quad |\xi|\geq\xi_T .
\]
Since \(0\leq\rho_k(\xi)\leq1\) and the projectors \(\Pi_k\) are orthogonal, this yields
\[
\widehat q_T(\xi)^*\mathcal G_T(\xi)\widehat q_T(\xi)
\geq
c_{\rm mh}T
\sum_{k=0}^K
\rho_k(\xi)\|\Pi_k\widehat q_T(\xi)\|_{\C^N}^2 .
\]
Since \(T\geq \min\{T,T^{2K+1}\}\), the two estimates imply
\begin{equation}\label{eq:global-frequency-bound-final}
\widehat q_T(\xi)^*\mathcal G_T(\xi)\widehat q_T(\xi)
\geq
c_*
\min\{T,T^{2K+1}\}
\sum_{k=0}^K
\rho_k(\xi)\|\Pi_k\widehat q_T(\xi)\|_{\C^N}^2,
\quad \xi\in\mathbb R ,
\end{equation}
where
$c_*:=\min\{c_{\rm low},c_{\rm mh}\}>0 .$
Integrating \eqref{eq:global-frequency-bound-final} in \(\xi\) and using
Plancherel's theorem, with the Plancherel constant absorbed into \(c_*\), gives
\[
\int_0^T\int_{\mathbb R}|B^\top q(t,x)|^2\,dx\,dt
\geq
c_*\min\{T,T^{2K+1}\}
\|q_T\|_{\mathcal Y_K^-}^2 .
\]
Equivalently,
\[
\|q_T\|_{\mathcal Y_K^-}^2
\leq
C_*
\max\left\{\frac1T,\frac1{T^{2K+1}}\right\}
\int_0^T\int_{\mathbb R}
|B^\top q(t,x)|^2\,dx\,dt .
\]
This proves the upper bound in the whole-line observability estimate
\eqref{eq:obs-global}.

We now show that this time dependence is sharp. Let \(C_{\rm obs}(T)\) denote
the best constant in
\[
\|q_T\|_{\mathcal Y_K^-}^2
\leq
C_{\rm obs}(T)
\int_0^T\int_{\mathbb R}|B^\top q(t,x)|^2\,dx\,dt .
\]
The preceding estimate gives
\[
C_{\rm obs}(T)
\lesssim
\max\left\{\frac1T,\frac1{T^{2K+1}}\right\}.
\]
We prove the reverse bound.
First, choose \(0\neq v_0\in\mathcal V_0^{\mathbb C}\) and terminal data
\[
\widehat q_T(\xi)=\varphi(\xi)v_0,
\quad
0\neq\varphi\in C_c^\infty(\mathbb R).
\]
Since the propagator \(e^{is\xi A^\top}\) is uniformly bounded for
\(\xi\in\operatorname{supp}\varphi\) and \(s\in[0,T]\), and since
\(\rho_0\equiv1\), we have
\[
\int_0^T\int_{\mathbb R}|B^\top q(t,x)|^2\,dx\,dt
\leq
CT\|q_T\|_{\mathcal Y_K^-}^2 .
\]
Hence
\[
C_{\rm obs}(T)\gtrsim T^{-1}.
\]

Next, choose \(0\neq v_K\in\mathcal V_K^{\mathbb C}\). By the triangular
Kalman relations,
\[
B^\top(A^\top)^\ell v_K=0,\qquad \ell<K.
\]
Thus, for \(T|\xi|\) small,
\[
B^\top e^{is\xi A^\top}v_K
=
(i\xi)^K\frac{s^K}{K!}B^\top(A^\top)^Kv_K
+
O(|\xi|^{K+1}s^{K+1}).
\]
Consequently,
\[
\int_0^T
\|B^\top e^{is\xi A^\top}v_K\|^2\,ds
\leq
C T^{2K+1}|\xi|^{2K}\|v_K\|^2 .
\]
Take terminal data
\[
\widehat q_T(\xi)=\varphi(\xi)v_K,
\quad
0\neq\varphi\in C_c^\infty(-\delta,\delta),
\]
with \(\delta>0\) sufficiently small. For \(0<T\leq T_0\), with \(T_0\delta\)
small enough, the preceding estimate gives
\[
\int_0^T\int_{\mathbb R}|B^\top q(t,x)|^2\,dx\,dt
\leq
C T^{2K+1}
\int_{\mathbb R}|\xi|^{2K}|\varphi(\xi)|^2\,d\xi .
\]
On the other hand, after choosing \(\delta\leq1\),
\[
\|q_T\|_{\mathcal Y_K^-}^2
=
\int_{\mathbb R}\rho_K(\xi)|\varphi(\xi)|^2\|v_K\|^2\,d\xi
\asymp
\int_{\mathbb R}|\xi|^{2K}|\varphi(\xi)|^2\,d\xi .
\]
Therefore,
\[
C_{\rm obs}(T)\gtrsim T^{-(2K+1)},
\quad 0<T\leq T_0 .
\]
Together with the previous lower bound \(C_{\rm obs}(T)\gtrsim T^{-1}\),
this yields
\[
C_{\rm obs}(T)
\gtrsim
\max\left\{\frac1T,\frac1{T^{2K+1}}\right\}.
\]
Combining this reverse bound with the upper bound proves the sharp time
dependence of the whole-line observability constant.

\section{Localized observability and frequency decomposition}
\label{sec:localized-observability}
We now pass from whole-line observation to observation on a set
\(\mathcal O\) satisfying \eqref{eq:gcc}. The low-frequency algebra has already
been identified in Section~3: the finite-time Gramian selects the
Kalman-adapted norm \(\mathcal Y_K^{-}\). The new difficulty is geometric.

The proof separates two regimes. Away from the origin, the Kalman weights are
uniformly equivalent to the standard \(L^2\) weights. In this regime, the
characteristic families can be separated by finite-order differential operators,
and each eigenspace component is recovered from \(B^\top q\) by the Hautus
injectivity condition, provided \(T>T_{\mathrm{geom}}\). Near the origin,
the geometry of the thick set must instead be combined with the whole-line
Kalman-weighted Gramian estimate and a Logvinenko--Sereda inequality. The final
localized estimate is obtained by recombining the two regimes, using the
saturation of the Kalman weights
$\rho_k(\xi)=\min\{|\xi|^{2k},1\}.$

Throughout this section we use the notation of \eqref{eq:gcc}. We also use the
Hautus form of the Kalman condition,
\[
\ker B^\top\cap\ker(A^\top-\lambda I)=\{0\},
\quad \lambda\in\sigma(A),
\]
or equivalently, \(B^\top\) is injective on each eigenspace of \(A^\top\).

\subsection{Geometry of thick observation sets}
\label{thick-sets}
We first record the geometric consequence of the thickness-and-gap condition needed in the characteristic recovery argument. This will be used in the propagation estimate below.
\begin{lemma}
\label{lem:density-bands}
Assume \eqref{eq:gcc}. Then, for every bounded interval
$I\subset\mathbb R,$
\[
|I\cap\mathcal O|
\ge
\frac{\ell_{\mathcal O}}{\ell_{\mathcal O}+G_{\mathcal O}}\bigl(|I|-G_{\mathcal O}\bigr)_+,\quad \text{where}
\quad (\cdot)_+ \quad \text{denotes the positive part.}
\]
\end{lemma}
\begin{proof}
Set $F:=\mathbb R\setminus\mathcal O.$ Up to sets of measure zero, $F$ is the
union of the gaps $(b_j,a_{j+1}),\quad j\in\mathbb Z,$
each of length at most $G_{\mathcal O}.$
Let $I\subset\mathbb R$ be a bounded interval. Let $n$ be the number of gaps
which intersect $I.$ If $n=0,$ then $|I\cap F|=0,$ and the estimate is
immediate.
Assume now that $n\ge1.$ Set
$e:=|I\cap\mathcal O|,
\quad
f:=|I\cap F|.$
Since each gap has length at most $G_{\mathcal O},$ we have
$f\le nG_{\mathcal O}.$
Moreover, between any two consecutive gaps met by $I,$ the interval $I$
contains at least one full connected component of $\mathcal O,$ whose length is at
least $\ell_{\mathcal O}.$ Therefore
$e\ge (n-1)\ell_{\mathcal O}.$
Hence
$n\le \frac{e}{\ell_{\mathcal O}}+1.$
Combining this with \(f\le nG_{\mathcal O}\), we obtain
\[
f\le G_{\mathcal O}\left(\frac{e}{\ell_{\mathcal O}}+1\right).
\]
Since $|I|=e+f,$ it follows that
\[
|I|
\le
e+G_{\mathcal O}\left(\frac{e}{\ell_{\mathcal O}}+1\right)
=
\frac{\ell_{\mathcal O}+G_{\mathcal O}}{\ell_{\mathcal O}}e+G_{\mathcal O}.
\quad 
\text{Thus}
\quad 
e\ge
\frac{\ell_{\mathcal O}}{\ell_{\mathcal O}+G_{\mathcal O}}\bigl(|I|-G_{\mathcal O}\bigr).
\]
If $|I|\le G_{\mathcal O},$ the right-hand side is nonpositive, and the estimate follows
after taking the positive part. Therefore
\[
|I\cap\mathcal O|
\ge
\frac{\ell_{\mathcal O}}{\ell_{\mathcal O}+G_{\mathcal O}}\bigl(|I|-G_{\mathcal O}\bigr)_+.
\]
\end{proof}
\subsection{Characteristic propagation and recovery estimates}
\label{subsec:characteristic}
The next lemma converts the geometric condition \(T>G_{\mathcal O}/|\mu|\)
into a quantitative recovery estimate for a profile transported with speed
\(\mu\). The proof is given in Appendix~\ref{lemm}. We emphasize that the negative
Sobolev estimate is not obtained by a formal dualization of the \(L^2\)
bound. It is proved by constructing a uniformly bounded right inverse of the
pull-back operator along the characteristics. The loss
\((T-T_\mu^*)^{-s-\frac12}\) comes from the lower bound on the time spent by
each characteristic in the thick observation set, together with the derivatives
of the associated cutoffs.
\begin{lemma}
\label{lem:quantitative}
Let \(d\ge1\), let \(\mu\neq0\), and assume that \(\mathcal O\) satisfies
\eqref{eq:gcc}. Set
\[
T_\mu^*:=\frac{G_{\mathcal O}}{|\mu|}.
\]
Let \(T_1>T_\mu^*\). Then, for every integer
\(s\ge0\), there exists a constant \(C_s(T_1)>0\), depending only on
\(s,d,\mu,\ell_{\mathcal O},G_{\mathcal O}\), and \(T_1\), such that, for every
\(T\in(T_\mu^*,T_1]\), with \(U_T:=(0,T)\times\mathcal O\), every
\(g\in H^{-s}(\mathbb R)^d\) satisfies
\begin{equation}
\label{eq:recovery}
\|g\|_{H^{-s}(\mathbb R)^d}
\le
\frac{C_s(T_1)}{(T-T_\mu^*)^{s+\frac12}}
\left\|
g(x+\mu(T-t))
\right\|_{H^{-s}(U_T)^d}.
\end{equation}
Here \(H^{-s}(U_T)^d\) denotes the dual of \(H_0^s(U_T)^d\). In particular, for
\(g\in L^2(\mathbb R)^d\), one has
\begin{equation}
\label{eq:L2-characteristic-recovery}
\|g\|_{L^2(\mathbb R)^d}^2
\le
\frac{C_0(T_1)}{T-T_\mu^*}
\int_0^T\int_{\mathcal O}
|g(x+\mu(T-t))|^2\,dx\,dt.
\end{equation}
\end{lemma}
We now use Lemma~\ref{lem:quantitative} to reduce the multi-speed observation
\(B^\top q\) to finitely many single-speed recovery estimates.
\subsection{Differential separation and high-frequency recovery}
\label{subsec:separation}
The differential separation argument used below is inspired by the
operator-composition strategy systematically employed by Zuazua
\cite{ZUAZUA201621} for the robust observation of additive
superpositions of solutions to commuting constant-coefficient PDEs.
The underlying principle is to compose the operators associated with
all unwanted components, thereby annihilating them and retaining only
the target dynamics. Although the detailed examples in \cite{ZUAZUA201621} mainly involve two superposed dynamics, the
argument is explicitly stated there for finite families of commuting
PDE operators.

In the present setting, we adapt this principle to the characteristic
families of a single partially observed transport system. More
precisely, for each characteristic speed $\lambda\in\sigma(A)$, we
compose the transport operators associated with all the remaining
speeds.
Since $A^\top$ is diagonalizable over $\mathbb R$, we write
\[
\mathbb R^N=\bigoplus_{\lambda\in\sigma(A)}E_\lambda,
\qquad
E_\lambda=\ker(A^\top-\lambda I).
\]
By the Hautus criterion, the Kalman rank condition is equivalent to
$\ker B^\top\cap E_\lambda=\{0\},
\quad \lambda\in\sigma(A).$
Thus $B^\top$ is injective on each characteristic eigenspace.
Let $p=\#\sigma(A)$ and, for $\lambda\in\sigma(A)$, define
\[
\mathcal D_\lambda
:=
\prod_{\substack{\nu\in\sigma(A)\\ \nu\neq\lambda}}
(\partial_t+\nu\partial_x).
\] 
Each factor annihilates the characteristic family of speed $\nu$; hence
$\mathcal D_\lambda$ separates the family of speed $\lambda$. Formally,
\[
\mathcal D_\lambda(B^\top q)
=
c_\lambda B^\top \partial_x^{p-1}
q_T^{(\lambda)}(x+\lambda(T-t)),
\quad
c_\lambda
:=
\prod_{\substack{\nu\in\sigma(A)\\ \nu\neq\lambda}}
(\nu-\lambda)\neq0.
\]
Thus the separation costs $p-1$ spatial derivatives.

For a fixed observation time \(T\), set
\(U:=U_T=(0,T)\times\mathcal O\) and \(r=p-1\). Here \(H_0^r(U)\) denotes the closure of \(C_c^\infty(U)\) in \(H^r(\mathbb R^2)\), and \(H^{-r}(U)\) is its Hilbert dual. The preceding identity is used in the weak sense: since
\(B^\top q\in L^2(U)^m\) and \(D_\lambda\) has order \(r\), one has
\[
D_\lambda(B^\top q)\in H^{-r}(U)^m.
\]
All identities involving \(D_\lambda(B^\top q)\) are understood in
\(\mathcal D'(U)^m\) and estimated in the \(H^{-r}(U)^m\)-norm.
The derivative loss is later removed on the high-frequency region by the
cutoff \(P_{\geq\rho}\).
\begin{proposition}
\label{prop:direct-recovery}
Assume that \(A^\top\) is diagonalizable over \(\mathbb R\), that \(0\notin\sigma(A)\), that \((A,B)\) satisfies the Kalman rank condition, and that \(\mathcal O\) satisfies \eqref{eq:gcc}. Let \(\lambda_1,\dots,\lambda_p\) be the distinct eigenvalues of \(A^\top\), and let \(T_{\rm geom}\) be the corresponding geometric time. For \(\rho>0\), define
\(P_{\ge\rho}\) by
\[
\widehat{P_{\ge\rho}u}(\xi)
=
\mathbf 1_{\{|\xi|\ge\rho\}}\widehat u(\xi).
\]
Then, for every \(T_1>T_{\rm geom}\), there exists
\(C_p(T_1,\rho)>0\), depending only on \(A\), \(B\), \(\mathcal O\), \(\rho\),
and \(T_1\), such that, for every \(T\in(T_{\rm geom},T_1]\) and every
\(q_T\in L^2(\mathbb R)^N\), the corresponding solution of
\eqref{adjoint-system} satisfies
\begin{equation}
\label{eq:localized-observability}
\|P_{\ge\rho}q_T\|_{L^2(\mathbb R)^N}^2
\le
\frac{C_p(T_1,\rho)}{(T-T_{\mathrm{geom}})^{2p-1}}
\int_0^T\int_\mathcal O |B^\top q(t,x)|^2\,dx\,dt.
\end{equation}
\end{proposition}
\begin{proof}
Fix \(T_1>T_{\rm geom}\), and let \(T\in(T_{\rm geom},T_1]\). Set
\[
r:=p-1,\quad U_T:=(0,T)\times\mathcal O,
\quad y:=(B^\top q)|_{U_T}.
\]
We first prove the identities below for terminal data
\(q_T\in\mathcal S(\mathbb R)^N\). In this case all characteristic
decompositions and differentiations are classical.

The extension to \(q_T\in L^2(\mathbb R)^N\) follows by density. Indeed, let
\(q_T^n\in\mathcal S(\mathbb R)^N\) be such that
\(q_T^n\to q_T\) in \(L^2(\mathbb R)^N\), and let \(q^n\) be the corresponding
adjoint solution. The observation map
$q_T\longmapsto B^\top q|_U$
is continuous from \(L^2(\mathbb R)^N\) to \(L^2(U)^m\). Hence
\(B^\top q^n\to B^\top q\) in \(L^2(U)^m\). Moreover, every differential
operator considered below is a constant-coefficient operator of order at most
\(r\), and therefore maps \(L^2(U)^m\) continuously into \(H^{-r}(U)^m\).
Consequently, the separated identities obtained for \(q_T^n\) pass to the
limit in \(H^{-r}(U)^m\), and hence in \(\mathcal D'(U)^m\).

\medskip
\noindent\emph{Step 1. Separation of one characteristic family.}

Since \(A^\top\) is diagonalizable over \(\mathbb R\), let
\(\mathcal P_{\lambda_\ell}\) be the spectral projector associated with
\(E_{\lambda_\ell}\). Thus
\[
I_N=\sum_{\ell=1}^p \mathcal P_{\lambda_\ell},
\quad
A^\top \mathcal P_{\lambda_\ell}=\lambda_\ell \mathcal P_{\lambda_\ell} .
\quad 
\text{Set}\quad 
q_T^{(\ell)}:=\mathcal P_{\lambda_\ell} q_T.
\quad \text{Then}
\]
\[
y(t,x)
:=B^\top q(t,x)=\sum_{\ell=1}^p Y_{\ell}(t,x)
=\sum_{\ell=1}^p
B^\top q_T^{(\ell)}(x+\lambda_\ell(T-t))
\quad
\text{in }L^2(U)^m.
\]

Fix \(\ell_0\in\{1,\dots,p\}\). Define the differential separation operator
\[
\mathcal D_{\lambda_{\ell_0}}
:=
\prod_{\ell\neq\ell_0}
(\partial_t+\lambda_\ell\partial_x),
\]
with the convention that \(\mathcal D_{\lambda_{\ell_0}}=\mathrm{Id}\) if \(p=1\). Since the factors have constant coefficients, they commute. Moreover, for every \(j,\ell\),
\[
(\partial_t+\lambda_\ell\partial_x)Y_j
=
(\lambda_\ell-\lambda_j)
B^\top\partial_x q_T^{(j)}(x+\lambda_j(T-t)).
\quad 
\text{In particular,}
\quad
(\partial_t+\lambda_j\partial_x)Y_j=0.
\]
Hence, if \(j\neq \ell_0\), the product defining \(\mathcal D_{\lambda_{\ell_0}}\) contains the annihilating factor
\(\partial_t+\lambda_j\partial_x\), and therefore
$\mathcal D_{\lambda_{\ell_0}}Y_j=0.$
On the other hand, for \(j=\ell_0\), each factor contributes
\((\lambda_\ell-\lambda_{\ell_0})\partial_x\). Thus
\[
\mathcal D_{\lambda_{\ell_0}}Y_{\ell_0}
=
c_{\ell_0}
B^\top\partial_x^r q_T^{(\ell_0)}
(x+\lambda_{\ell_0}(T-t)),
\quad 
\text{with}
\quad 
c_{\ell_0}
:=
\prod_{\ell\neq\ell_0}
(\lambda_\ell-\lambda_{\ell_0})\neq0.
\]
Hence
\begin{equation}\label{eq:isolated-family}
\mathcal D_{\lambda_{\ell_0}}y
=
c_{\ell_0}
B^\top\partial_x^r q_T^{(\ell_0)}
(x+\lambda_{\ell_0}(T-t))
\quad
\text{in }\mathcal D'(U)^m .
\end{equation}
\medskip
\noindent\emph{Step 2. Recovery along the observed characteristic tube.}
Expanding
\[
\mathcal D_{\lambda_{\ell_0}}
=
\prod_{\ell\neq\ell_0}(\partial_t+\lambda_\ell\partial_x)
\]
gives a linear combination of derivatives
$\partial_t^{r-j}\partial_x^j,\qquad 0\le j\le r,$
with coefficients depending only on the speeds
\(\lambda_\ell\), \(\ell\neq\ell_0\). Hence
\(\mathcal D_{\lambda_{\ell_0}}\) is continuous from \(L^2(U)^m\) to
\(H^{-r}(U)^m\). 
Therefore, there exists \(S_{\ell_0}>0\) such that
\begin{equation}\label{eq:operator-nega}
\|\mathcal D_{\lambda_{\ell_0}}y\|_{H^{-r}(U)^m}
\le
S_{\ell_0}\|y\|_{L^2(U)^m}.
\end{equation}
Combining \eqref{eq:operator-nega} with \eqref{eq:isolated-family} yields
\begin{equation}\label{eq:observed-tube}
\left\|
B^\top\partial_x^r q_T^{(\ell_0)}
(x+\lambda_{\ell_0}(T-t))
\right\|_{H^{-r}(U)^m}
\leq
\frac{S_{\ell_0}}{|c_{\ell_0}|}
\|B^\top q\|_{L^2(U)^m}.
\end{equation}
We use Lemma~\ref{lem:quantitative} in its negative Sobolev form, with
\[
s=r,\quad
\mu=\lambda_{\ell_0},\quad
g=B^\top\partial_x^r q_T^{(\ell_0)}\in H^{-r}(\mathbb R)^m.
\]
Since
\[
T_{\lambda_{\ell_0}}^*
:=
\frac{G_{\mathcal O}}{|\lambda_{\ell_0}|}
\leq
T_{\mathrm{geom}},
\]
we have \(T-T_{\lambda_{\ell_0}}^*\geq T-T_{\mathrm{geom}}>0\). Hence the lemma,
together with \eqref{eq:observed-tube}, gives
\begin{equation}\label{eq:B-derivative-recovered}
\left\|
B^\top\partial_x^r q_T^{(\ell_0)}
\right\|_{H^{-r}(\mathbb R)^m}
\leq
\frac{
C_r^{\lambda_{\ell_0}}(T_1)S_{\ell_0}
}
{
|c_{\ell_0}|(T-T_{\mathrm{geom}})^{r+\frac{1}{2}}
}
\|B^\top q\|_{L^2(U)^m},
\end{equation}
where \(C_r^{\lambda_{\ell_0}}(T_1)\) denotes the constant in
Lemma~\ref{lem:quantitative}, corresponding to \(\mu=\lambda_{\ell_0}\), valid
uniformly for \(T\in(T_{\rm geom},T_1]\).

Finally, set
\[
q_T^{\geq\rho,(\ell_0)}
:=
P_{\geq\rho}q_T^{(\ell_0)}
.
\]
Since \(P_{\geq\rho}\) is a scalar Fourier multiplier, it commutes with
\(\partial_x\), and \(B^\top\). Moreover,
\(P_{\geq\rho}\) is a contraction on \(H^{-r}(\mathbb R)\). Thus
\begin{equation}\label{eq:B-derivative-cutoff}
\left\|
B^\top\partial_x^r q_T^{\geq\rho,(\ell_0)}
\right\|_{H^{-r}(\mathbb R)^m}
\leq
\frac{
C_r^{\lambda_{\ell_0}}(T_1)S_{\ell_0}
}
{
|c_{\ell_0}|(T-T_{\mathrm{geom}})^{r+\frac{1}{2}}
}
\|B^\top q\|_{L^2(U)^m}.
\end{equation}
\medskip
\noindent\emph{Step 3. Removal of the observation matrix by Hautus injectivity.}
We now remove \(B^\top\). By the Hautus criterion, the Kalman rank condition gives
$\ker(B^\top)\cap E_{\lambda_{\ell_0}}=\{0\}.$
Hence \(B^\top\) is injective on \(E_{\lambda_{\ell_0}}\). Its complexification
is therefore injective on
\[
E_{\lambda_{\ell_0}}^{\mathbb C}
:=
E_{\lambda_{\ell_0}}\otimes_{\mathbb R}\mathbb C
=
\{v_1+iv_2:\ v_1,v_2\in E_{\lambda_{\ell_0}}\}.
\]
Since \(E_{\lambda_{\ell_0}}^{\mathbb C}\) is finite-dimensional, there exists
\[
\alpha_{\ell_0}
:=
\min_{\substack{v\in E_{\lambda_{\ell_0}}^{\mathbb C}\\ |v|=1}}
|B^\top v|>0,
\quad
\text{and hence}
\quad 
|v|\leq \alpha_{\ell_0}^{-1}|B^\top v|,
\quad
v\in E_{\lambda_{\ell_0}}^{\mathbb C}.
\quad 
\text{Set}
\quad 
u_{\ell_0}:=\partial_x^r q_T^{\geq\rho,(\ell_0)}.
\]
Since \(q_T^{\geq\rho,(\ell_0)}\) is the \(\lambda_{\ell_0}\)-spectral
component of \(P_{\geq\rho}q_T\), we have
\[
\widehat u_{\ell_0}(\xi)\in E_{\lambda_{\ell_0}}^{\mathbb C}
\quad\text{for a.e. }\xi.
\]
Equivalently,
\[
\widehat u_{\ell_0}(\xi)
=
\widehat u_{\ell_0}^{\,1}(\xi)
+
i\widehat u_{\ell_0}^{\,2}(\xi),
\quad
\widehat u_{\ell_0}^{\,1}(\xi)\hbox{, }
\widehat u_{\ell_0}^{\,2}(\xi)
\in E_{\lambda_{\ell_0}}
\quad\text{for a.e. }\xi.
\]
Therefore,
\[
|\widehat u_{\ell_0}(\xi)|
\leq
\alpha_{\ell_0}^{-1}
|B^\top\widehat u_{\ell_0}(\xi)|
\quad\text{for a.e. }\xi.
\]
Multiplying by \((1+|\xi|^2)^{-r/2}\), squaring, and integrating in \(\xi\),
we obtain
\[
\|u_{\ell_0}\|_{H^{-r}(\mathbb R)^N}
\leq
\alpha_{\ell_0}^{-1}
\|B^\top u_{\ell_0}\|_{H^{-r}(\mathbb R)^m}.
\]
Combining this with \eqref{eq:B-derivative-cutoff}, we obtain
\begin{equation}\label{eq:derivative-recovered}
\left\|
\partial_x^r q_T^{\geq\rho,(\ell_0)}
\right\|_{H^{-r}(\mathbb R)^N}
\leq
\frac{
C_r^{\lambda_{\ell_0}}(T_1)S_{\ell_0}
}
{
\alpha_{\ell_0}|c_{\ell_0}|(T-T_{\mathrm{geom}})^{r+\frac{1}{2}}
}
\|B^\top q\|_{L^2(U)^m}.
\end{equation}
\medskip
\noindent\emph{Step 4. Inversion of the derivative on high frequencies and summation over modes.}

It remains to recover \(q_T^{\geq\rho,(\ell_0)}\) from its \(r\)-th
derivative. If \(\operatorname{supp}\widehat u\subset\{|\xi|\geq\rho\}\), then
\[
\|\partial_x^r u\|_{H^{-r}}^2
=
\int_{\mathbb R}
\left(\frac{|\xi|^2}{1+|\xi|^2}\right)^r
|\widehat u(\xi)|^2\,d\xi
\geq
\left(\frac{\rho^2}{1+\rho^2}\right)^r
\|u\|_{L^2}^2.
\]
Thus
\[
\|u\|_{L^2}
\leq
\left(\frac{1+\rho^2}{\rho^2}\right)^{r/2}
\|\partial_x^r u\|_{H^{-r}}.
\]
Applying this to \(u=q_T^{\geq\rho,(\ell_0)}\) and using
\eqref{eq:derivative-recovered}, we obtain
\begin{equation}\label{eq:component}
\|q_T^{\geq\rho,(\ell_0)}\|_{L^2(\mathbb R)^N}
\leq
\frac{
K_{\ell_0}(T_1,\rho)
}
{
(T-T_{\mathrm{geom}})^{r+\frac{1}{2}}
}
\|B^\top q\|_{L^2(U)^m}, \quad \text{where}\quad K_{\ell_0}(T_1,\rho)
:=
\left(\frac{1+\rho^2}{\rho^2}\right)^{r/2}
\frac{
C_r^{\lambda_{\ell_0}}(T_1)S_{\ell_0}
}
{
\alpha_{\ell_0}|c_{\ell_0}|
}.
\end{equation}
Since the spectral projectors need not be orthogonal, we use only the
triangle inequality. Namely,
\[
P_{\geq\rho}q_T
=
\sum_{\ell=1}^p q_T^{\geq\rho,(\ell)}.
\]
Therefore, using \eqref{eq:component} for each \(\ell\), and recalling that
\(r=p-1\), we get
\[
\|P_{\geq\rho}q_T\|_{L^2(\mathbb R)^N}
\leq
\sum_{\ell=1}^p
\|q_T^{\geq\rho,(\ell)}\|_{L^2(\mathbb R)^N}
\leq
\frac{
\sum_{\ell=1}^p K_\ell(T_1,\rho)
}
{
(T-T_{\mathrm{geom}})^{p-\frac{1}{2}}
}
\|B^\top q\|_{L^2(U)^m}.
\]
Squaring gives
\[
\|P_{\geq\rho}q_T\|_{L^2(\mathbb R)^N}^2
\leq
\frac{
C_p(T_1,\rho)
}
{
(T-T_{\mathrm{geom}})^{2p-1}
}
\int_0^T\int_\mathcal O |B^\top q(t,x)|^2\,dx\,dt,
\quad 
\text{with}
\quad 
C_p(T_1,\rho)
:=
\left(\sum_{\ell=1}^p K_\ell(T_1,\rho)\right)^2,
\]
for some constant \(C_p(T_1,\rho)>0\) depending only on  \(A,B,\mathcal O,\rho\), and \(T_1\).
This proves \eqref{eq:localized-observability}.
\end{proof}
Proposition~\ref{prop:direct-recovery} completes the nonzero-frequency part of the localized argument. It shows that, once \(|\xi|\ge\rho\), the geometric threshold and the Hautus injectivity are sufficient to recover the full terminal state from the original observation \(B^\top q\).

To extend the preceding $L^2$-based recovery estimate to terminal data in $\mathcal Y_K^-,$ we use a frequency truncation argument. We first record the admissibility of the localized observation map on \(\mathcal Y_K^-\).
\begin{lemma}
\label{lem:YK-admissibility}
Let $T>0$. The observation map
$q_T\longmapsto B^\top q$
is continuous from $\mathcal Y_K^-$ into $L^2((0,T)\times\mathcal O)^m$. More precisely,
there exists $C_{\mathrm{adm}}(T)>0$ such that
\[
\|B^\top q\|_{L^2((0,T)\times\mathcal O)^m}
\le
C_{\mathrm{adm}}(T)\|q_T\|_{\mathcal Y_K^-} .
\]
\end{lemma}

\begin{proof}
We first define the observation operator on \(\mathcal Y_K^-\). For
\(u\in \mathcal S(\mathbb R)^N\), the Kalman-layer orthogonality gives, for
\(k\ge1\),
\[
\mathcal V_k\subset \mathcal K_{k-1}^{\perp},
\qquad
\operatorname{Ran}B=\mathcal K_0\subset \mathcal K_{k-1}.
\]
Hence \(B^\top\mathcal V_k=\{0\}\) for every \(k\ge1\), and therefore
$B^\top u=B^\top\Pi_0u.$
Since the weight on \(\mathcal V_0\) in the \(\mathcal Y_K^-\)-norm is equal
to \(1\), we have
\[
\|\Pi_0u\|_{L^2(\mathbb R)^N}
\le
\|u\|_{\mathcal Y_K^-}.
\quad
\text{Consequently,}
\quad 
\|B^\top\Pi_0u\|_{L^2(\mathbb R)^m}
\le
C\|u\|_{\mathcal Y_K^-}.
\]
Thus the map \(u\mapsto B^\top u\), initially defined on
\(\mathcal S(\mathbb R)^N\), extends uniquely and continuously from
\(\mathcal Y_K^-\) to \(L^2(\mathbb R)^m\), and this extension is given by
$B^\top u:=B^\top\Pi_0u,
\quad u\in\mathcal Y_K^-.$

Using the boundedness of the homogeneous adjoint evolution on
\(\mathcal Y_K^-\), namely
\[
\sup_{0\le t\le T}\|q(t)\|_{\mathcal Y_K^-}
\le
C_{\mathrm{evo}}(T)\|q_T\|_{\mathcal Y_K^-},
\]
we obtain
\[
\begin{aligned}
\|B^\top q\|_{L^2((0,T)\times\mathcal O)^m}^2
&\le
\int_0^T
\|B^\top q(t)\|_{L^2(\mathbb R)^m}^2\,dt  \\
&\le
CT\,C_{\mathrm{evo}}(T)^2
\|q_T\|_{\mathcal Y_K^-}^2 .
\end{aligned}
\]
This proves the admissibility estimate.
\end{proof}
Consequently, for every \(T_1>0\), the admissibility constant can be chosen uniformly for \(0<T\le T_1\): there exists \(C_{\mathrm{adm}}(T_1)>0\) such that
\[
\|B^\top q\|_{L^2((0,T)\times\mathcal O)^m}
\le
C_{\mathrm{adm}}(T_1)\|q_T\|_{\mathcal Y_K^-},
\qquad 0<T\le T_1 .
\]
For \(q_T\in\mathcal Y_K^-\), the observation \(B^\top q\) is understood as the continuous extension of the classical observation map defined on \(\mathcal S(\mathbb R)^N\): if \(q_T^n\to q_T\) in \(\mathcal Y_K^-\), with \(q_T^n\in\mathcal S(\mathbb R)^N\), then \(B^\top q^n\) converges in \(L^2((0,T)\times\mathcal O)^m\), and the limit is independent of the approximating sequence. We still denote this limit by \(B^\top q\).
We now combine Proposition~\ref{prop:direct-recovery} with the admissibility
lemma to extend the estimate from \(L^2\) terminal data to arbitrary data in
\(\mathcal Y_K^-\).
\begin{corollary}\label{recovvv}
Under the assumptions of Proposition~\ref{prop:direct-recovery}, let
\(\rho>0\) and \(T_1>T_{\mathrm{geom}}\). Then there exists a constant
\(C_p(T_1,\rho)>0\), depending only on \(A,B,\mathcal O,\rho\), and \(T_1\),
such that, for every \(T\in(T_{\mathrm{geom}},T_1]\), every
\(q_T\in \mathcal Y_K^-\), and every corresponding \(\mathcal Y_K^-\)-solution
\(q\), one has
\begin{equation}
\|P_{\ge\rho}q_T\|_{L^2(\mathbb R)^N}^2
\le
\frac{C_p(T_1,\rho)}
{(T-T_{\mathrm{geom}})^{2p-1}}
\int_0^T\int_\mathcal O |B^\top q(t,x)|^2\,dx\,dt .
\end{equation}
Here $p=\#\sigma(A)$ is the number of distinct characteristic speeds.
\end{corollary}

\begin{proof}
Let $q_T\in \mathcal Y_K^-,$ and let $q$ be the corresponding homogeneous adjoint
solution. Fix $0<\varepsilon<\rho,$ and set
$q_{T,\varepsilon}:=P_{|\xi|\ge\varepsilon}q_T.$
Since the weights defining $\mathcal Y_K^-$ are bounded below on
$\{|\xi|\ge\varepsilon\},$ we have
$q_{T,\varepsilon}\in L^2(\mathbb R)^N.$
Let \(q_\varepsilon\) be the solution with terminal datum \(q_{T,\varepsilon}\).
Because the equation has constant coefficients,
$q_\varepsilon=P_{|\xi|\ge\varepsilon}q.$
Applying Proposition~\ref{prop:direct-recovery} to \(q_\varepsilon\)
gives
\[
\|P_{\ge\rho}q_{T,\varepsilon}\|_{L^2(\mathbb R)^N}^2
\le
\frac{C_p(T_1,\rho)}{(T-T_{\mathrm{geom}})^{2p-1}}
\int_0^T\int_\mathcal O |B^\top q_\varepsilon(t,x)|^2\,dx\,dt .
\]
Since \(0<\varepsilon<\rho\), we have
\[
P_{\ge\rho}q_{T,\varepsilon}=P_{\ge\rho}q_T.
\quad 
\text{Moreover,}
\quad
q_{T,\varepsilon}\to q_T
\qquad\text{in }\mathcal Y_K^-
\]
as \(\varepsilon\downarrow0\). By Lemma~\ref{lem:YK-admissibility},
\[
B^\top q_\varepsilon\to B^\top q
\qquad\text{in }L^2((0,T)\times\mathcal O)^m.
\]
Passing to the limit $\varepsilon\downarrow0$ gives
\[
\|P_{\ge\rho}q_T\|_{L^2(\mathbb R)^N}^2
\le
\frac{C_p(T_1,\rho)}
{(T-T_{\mathrm{geom}})^{2p-1}}
\int_0^T\int_\mathcal O |B^\top q(t,x)|^2\,dx\,dt,
\]
which proves the claim.
\end{proof}
\subsection{Low-frequency localization}
\label{low-frequency-localization}
We now localize the bounded-frequency part by combining the whole-line Kalman-weighted observability estimate with the Logvinenko--Sereda inequality.
\begin{proposition}
\label{prop:bounded-frequency-localization}
Assume that \(A^\top\) is diagonalizable over \(\mathbb R\), that \((A,B)\) satisfies 
the Kalman rank condition, and that \(\mathcal O\) satisfies 
\eqref{eq:gcc}. Let \(\rho>0\). Then, for every 
\(T>0\), there exists \(C_{LF}(T,\rho)=C(A,B,K,T,\mathcal O,\rho)>0\) such that, for every 
\(q_T\in\mathcal Y_K^-\), if
$q_T^{LF}:=P_{|\xi|\le \rho}q_T$
and if \(q^{LF}\) denotes the corresponding adjoint solution, then
\[
\|q_T^{LF}\|_{\mathcal Y_K^-}^2
\leq
C_{LF}(T,\rho)
\int_0^T\int_{\mathcal O}
|B^\top q^{LF}(t,x)|^2\,dx\,dt .
\]
\end{proposition}
\begin{proof}
Since the adjoint equation has constant coefficients, the Fourier support is preserved 
by the flow. Hence, for every $t\in(0,T),$
\[
\operatorname{supp}\widehat{B^\top q^{LF}}(t,\cdot)
\subset [-\rho,\rho].
\]
By Lemma~\ref{lem:density-bands}, the set \(\mathcal O\) is 
\((\gamma_0,L_0)\)-thick, with
\[
L_0:=2G_{\mathcal O}+\ell_{\mathcal O},
\quad
\gamma_0:=\frac{\ell_{\mathcal O}}{2G_{\mathcal O}+\ell_{\mathcal O}}.
\]
Therefore, the Logvinenko--Sereda inequality 
\cite{LogvinenkoSereda1974,Kovrijkine2001} gives
\[
\|f\|_{L^2(\mathbb R)}^2
\leq
C_{LS}(\rho,\mathcal O)
\|f\|_{L^2(\mathcal O)}^2
\]
for every $f\in L^2(\mathbb R)$ such that 
$\operatorname{supp}\widehat f\subset[-\rho,\rho].$ Applying this estimate 
componentwise to \(f=B^\top q^{LF}(t,\cdot)\), and integrating in time, yields
\[
\int_0^T\int_{\mathbb R}
|B^\top q^{LF}(t,x)|^2\,dx\,dt
\leq
C_{LS}(\rho,\mathcal O)
\int_0^T\int_{\mathcal O}
|B^\top q^{LF}(t,x)|^2\,dx\,dt .
\]
On the other hand, the whole-line Kalman-weighted observability estimate gives
\[
\|q_T^{LF}\|_{\mathcal Y_K^-}^2
\leq
C_T
\int_0^T\int_{\mathbb R}
|B^\top q^{LF}(t,x)|^2\,dx\,dt .
\]
Combining the last two estimates gives the desired inequality, with
$C_{LF}(T,\rho)
:=
C_T\,C_{LS}(\rho,\mathcal O).$
For general data \(q_T\in\mathcal Y_K^-\), the estimate follows by density and by the 
admissibility of the localized observation map on \(\mathcal Y_K^-\).
\end{proof}
Consequently, in the proper localized case, \(G_{\mathcal O}>0\).
Since \(0\notin\sigma(A)\), one has \(T_{\rm geom}>0\). Hence, for every
\(T_1>T_{\rm geom}\), the whole-line low-frequency constant appearing in
Proposition 4.6 is uniformly bounded for \(T\in(T_{\rm geom},T_1]\).
Thus \(C_{LF}(T,\rho)\) can be replaced by a constant
\(C_{LF}(T_1,\rho)\) independent of \(T\).
\subsection{Recombination}
\label{subsec:recombination-localized}

We now combine the frequency-localized observability estimate with the recovery estimate
away from zero frequency.

\begin{proof}[Proof of Theorem~\ref{thm:local}]
The preceding estimates are valid for any fixed threshold \(\rho>0\). The constants in
the frequency-localized observability estimate may depend on \(\rho\), while the
nonzero-frequency recovery estimate is uniform on the region \(\{|\xi|\ge \rho\}\).

We choose \(\rho=1\), which is the natural saturation scale of the Kalman weights
$\rho_k(\xi)=\min\{|\xi|^{2k},1\}.$
Thus the region \(\{|\xi|\le 1\}\) carries the Kalman low-frequency weights, whereas
\(\{|\xi|>1\}\) corresponds to the saturated \(L^2\)-regime.
Fix $T>T_{\mathrm{geom}}.$ We decompose the terminal datum as
\[
q_T=q_T^{LF}+q_T^{HF},
\quad
q_T^{LF}:=P_{|\xi|\le1}q_T,
\quad
q_T^{HF}:=P_{|\xi|>1}q_T.
\]
Since the adjoint equation has constant coefficients, this decomposition is preserved
by the flow. Hence $q=q^{LF}+q^{HF}.$
We now use Corollary~\ref{recovvv} with \(\rho=1\), which extends
Proposition~\ref{prop:direct-recovery} from \(L^2(\mathbb R)^N\) terminal data
to terminal data in \(\mathcal Y_K^-\).
For simplicity, we denote the localized observation functional by
\[
O_T(q)
:=
\int_0^T\int_{\mathcal O}|B^\top q(t,x)|^2\,dx\,dt.
\]

By the definition of the Kalman-adapted norm and since the weights are saturated on
\(\{|\xi|>1\}\), we have
\[
    \|q_T\|_{\mathcal Y_K^-}^2
    =
    \|q_T^{LF}\|_{\mathcal Y_K^-}^2
    +
    \|q_T^{HF}\|_{L^2(\mathbb R)^N}^2.
\]
The low- and high-frequency parts are orthogonal in the global Fourier norm.
After restriction to the observation set \(\mathcal O\), this orthogonality is
no longer available.

The frequency-localized observability estimate gives
\[
    \|q_T^{LF}\|_{\mathcal Y_K^-}^2
    \le
    C_{LF}(T_1,1)
    \int_0^T\int_{\mathcal O}|B^\top q^{LF}(t,x)|^2\,dx\,dt.
\]
Since \(B^\top q^{LF}=B^\top q-B^\top q^{HF}\), we obtain
\[
    \|q_T^{LF}\|_{\mathcal Y_K^-}^2
    \le
    2C_{LF}(T_1,1)O_T(q)
    +
    2C_{LF}(T_1,1)
    \|B^\top q^{HF}\|_{L^2((0,T)\times\mathcal O)^m}^2.
\]
By the admissibility of the localized observation map on \(\mathcal Y_K^-\),
\[
    \|B^\top q^{HF}\|_{L^2((0,T)\times\mathcal O)^m}^2
    \le
    C_{\rm adm}(T_1)^2
    \|q_T^{HF}\|_{\mathcal Y_K^-}^2.
\]
On the high-frequency region \(\{|\xi|>1\}\), the \(\mathcal Y_K^-\)-norm
coincides with the \(L^2\)-norm. Therefore, by Corollary~\ref{recovvv} with
\(\rho=1\),
\[
\|q_T^{HF}\|_{\mathcal Y_K^-}^2
=
\|q_T^{HF}\|_{L^2(\mathbb R)^N}^2
\le
\frac{C_p(T_1,1)}
{(T-T_{\mathrm{geom}})^{2p-1}}O_T(q).
\]
Combining the previous inequalities, we get
\[
    \|q_T^{LF}\|_{\mathcal Y_K^-}^2
    \le
    2C_{LF}(T_1,1) O_T(q)
    +
    \frac{
    2C_{LF}(T_1,1)C_{\rm adm}(T_1)^2C_p(T_1,1)
    }
    {(T-T_{\mathrm{geom}})^{2p-1}}
     O_T(q).
\]
Adding the high-frequency estimate yields
\[
\|q_T\|_{\mathcal Y_K^-}^2
\le
\left[
2C_{LF}(T_1,1)
+
\frac{
C_p(T_1,1)+2C_{LF}(T_1,1)C_{\rm adm}(T_1)^2C_p(T_1,1)
}
{(T-T_{\mathrm{geom}})^{2p-1}}
\right] O_T(q).
\]
Equivalently, after enlarging the constant, we obtain
\[
\|q_T\|_{\mathcal Y_K^-}^2
\le
\frac{C(\mathcal O,A,B,K,T_1)}
{(T-T_{\mathrm{geom}})^{2p-1}}
\int_0^T\int_{\mathcal O}|B^\top q(t,x)|^2\,dx\,dt.
\]
Since \(T\le T_1\), all bounded constants depending on \(T_1\) can be absorbed
into the factor \((T-T_{\mathrm{geom}})^{-(2p-1)}\), after enlarging the final
constant.
This proves the localized observability estimate. The blow-up rate displayed above
comes from the nonzero-frequency recovery step; no optimality of the exponent
\(2p-1\) is claimed.
\end{proof}
\begin{remark}
\label{rem:loss-and-Jordan}
The differential separation argument used in
Proposition~\ref{prop:direct-recovery} produces an isotropic derivative loss of
order \(p-1\). If this loss is not inverted away from \(\xi=0\), the same
argument yields, for every \(\rho>0\),
\[
\int_{|\xi|\leq \rho}
|\xi|^{2(p-1)}|\widehat q_T(\xi)|^2\,d\xi
\lesssim_{T,\rho}
\frac{1}{(T-T_{\mathrm{geom}})^{2p-1}}
\int_0^T\int_{\mathcal O}|B^\top q(t,x)|^2\,dx\,dt .
\]
This estimate assigns the worst low-frequency loss
\(|\xi|^{2(p-1)}\) to all components. Since \(A^\top\) is diagonalizable with
\(p\) distinct eigenvalues, its minimal polynomial has degree \(p\), and hence
the Kalman index satisfies \(K\leq p-1\). The Kalman-adapted estimate is sharper:
it assigns the loss \(|\xi|^{2k}\) only to the corresponding layer
\(\mathcal V_k\).
The localized nonzero-frequency recovery also relies on the diagonalizability
of \(A^\top\): the proof separates pure characteristic families and then uses
the Hautus condition
$\ker B^\top\cap E_\mu=\{0\}$
to recover each eigenspace component. In the presence of Jordan blocks, the
components inside a generalized eigenspace remain coupled along the nilpotent
chain. Treating this case would require a triangular localized recovery
argument, which is outside the scope of the present paper.
\end{remark}
\section{Sharpness of the Kalman weights and optimal anisotropic spaces}\label{opti}
Sharpness issues are classical in observability and control theory: structural conditions are often both sufficient and necessary; see, e.g.,
\cite{BardosLebeauRauch1992}. Related optimality phenomena for algebraic coupling appear in indirectly actuated systems; see \cite{AlabauLeautaud2012}.
We now prove the sharpness of the functional framework introduced above. First, the Kalman condition is necessary for observability. Second, within the natural class of Hilbert norms defined by scalar Fourier multipliers on each Kalman layer adapted to the Kalman decomposition
$\mathbb R^N=\bigoplus_{k=0}^K\mathcal V_k,$
the low-frequency weights of $\mathcal Y_K^-$ are optimal. More precisely, on the $k$-th Kalman layer, the weight $|\xi|^{2k}$ cannot be replaced by any stronger low-frequency weight without destroying observability. Combined with the observability results of Sections~3 and~4, the preceding result proves that $\mathcal Y_K^-$ is optimal within the class of block-diagonal Fourier-multiplier spaces.
\begin{theorem}
\label{thm:sharpness}
\begin{enumerate}
\item \textbf{Failure of observability without the Kalman condition.}
Assume that
\[
\operatorname{rank}(B,AB,\dots,A^{N-1}B)=r<N.
\]
Then there exists a nontrivial subspace $\mathcal U\subset\R^N$ such that,
for any $q_T\in\mathcal U$, the corresponding solution satisfies
\[
B^\top e^{i\xi(T-t)A^\top} q_T = 0 \quad \text{for all } t\in[0,T].
\]
This subspace coincides with the orthogonal complement of the Kalman space $\mathcal K_{N-1}$.
Consequently, no observability inequality can hold in any norm that is nondegenerate on $\mathcal U$.
In particular, exact controllability on the whole state space fails.

\item \textbf{Optimality of the weights.}
Fix \(k\in\{0,\dots,K\}\) such that \(\mathcal V_k\neq\{0\}\), and let
\(w:\mathbb R\to[0,\infty)\) be a measurable weight satisfying
\[
\frac{w(\xi)}{|\xi|^{2k}}\xrightarrow[\xi\to0]{}+\infty.
\]
Then no constant \(C>0\) can exist such that, for every admissible terminal
datum \(q_T\) and corresponding adjoint solution \(q\),
\[
\int_{\mathbb R} w(\xi)\,
\big|\Pi_k\widehat q_T(\xi)\big|^2\,d\xi
\le
C\int_0^T\!\!\int_{\mathcal O} |B^\top q|^2\,dx\,dt.
\]
Equivalently, on the \(k\)-th nontrivial Kalman layer, the low-frequency
weight \(|\xi|^{2k}\) cannot be replaced by any weight which is uniformly
stronger near \(\xi=0\).
\end{enumerate}

Consequently, among Hilbert spaces whose norms are defined by block-diagonal Fourier multipliers adapted to the Kalman decomposition, the space $\mathcal Y_K^-$ is sharp: on each nontrivial Kalman layer, its low-frequency weight cannot be replaced by a uniformly stronger one while preserving observability.
\end{theorem}
The proof below shows sharpness by concentrating terminal data on a single Kalman layer and near the zero frequency.
\begin{proof}
\textbf{Failure of observability without the Kalman condition.}\\
Define
\[
\mathcal K_{N-1}:=\operatorname{Ran}(B,AB,\dots,A^{N-1}B),\qquad
\mathcal U:=\mathcal K_{N-1}^\perp\neq\{0\},
\]
and let $\Pi_{\mathcal U}$ denote the orthogonal projector onto $\mathcal U$.
By definition,
\[
v\in\mathcal U
\quad\Longleftrightarrow\quad
B^\top (A^\top)^\ell v=0,\qquad \ell=0,\dots,N-1.
\]
By the Cayley-Hamilton theorem, any power $(A^\top)^\ell$ with $\ell\ge N$
can be written as a linear combination of $(A^\top)^j$, $j=0,\dots,N-1$.
Therefore, the above property extends to all $\ell\ge0$.

Let $L\in\mathcal U$ and $\xi\in\R$. Solving the Fourier adjoint equation
$\partial_t\widehat q+i\xi A^\top\widehat q=0$ with terminal datum
$\widehat q(T,\xi)=L$ yields
\[
\widehat q(t,\xi)
= e^{i\xi(T-t)A^\top}L
= \sum_{\ell\ge0}\frac{(T-t)^\ell}{\ell!}(i\xi)^\ell (A^\top)^\ell L.
\]
Applying $B^\top$ and using $B^\top (A^\top)^\ell L=0$ gives
\[
B^\top e^{i\xi(T-t)A^\top}L \equiv 0,
\qquad \forall\,t\in[0,T],\ \forall\,\xi\in\R.
\]
Hence, if $q_T$ takes values in $\mathcal U$, then $\widehat q_T(\xi)\in\mathcal U$
for a.e.\ $\xi$, and therefore
\[
B^\top \widehat q(t,\xi)=0
\qquad \text{for a.e.\ }\xi\in\R,\ \forall\,t\in[0,T].
\]
By inverse Fourier transform,
$B^\top q\equiv0
\quad\text{on }(0,T)\times\R.$
At the PDE level, decompose $q_T=q_T^{\mathcal U}+q_T^{S}$ with values in
$\mathcal U$ and $\mathcal K_{N-1}$. By linearity,
\[
B^\top q^{\mathcal U}\equiv0
\quad\text{on }(0,T)\times\R,
\]
so that the observation detects only the component in $\mathcal K_{N-1}$.
If a norm $X$ is nondegenerate on $\mathcal U$, choosing
$q_T^{\mathcal U}\not\equiv0$ and $q_T^{S}\equiv0$ violates any observability inequality.
\\
\textbf{Optimality of the weights.}
Assume now that the Kalman rank condition holds. Fix \(k\in\{0,\dots,K\}\)
such that \(\mathcal V_k\neq\{0\}\), and assume that the weight
\(w:\R\to[0,\infty)\) satisfies
\[
\frac{w(\xi)}{|\xi|^{2k}}\xrightarrow[\xi\to0]{}+\infty .
\]
Let \(\mathcal V_k\) be a nontrivial Kalman block and let \(\Pi_k\) be the
orthogonal projector onto \(\mathcal V_k\). 
Using the result \eqref{eq:layerwise-asymptotics} of Corollary~\ref{cor:layerwise-asymptotics}, there exist
constants \(c_1,c_2>0\) and \(\xi_0>0\) such that, for all
\(v\in\mathcal V_k\) and all \(|\xi|\le \xi_0\),
\[
c_1|\xi|^{2k}|v|^2
\le
v^*\mathcal G_T(\xi)v
\le
c_2|\xi|^{2k}|v|^2 .
\]

Fix \(v\in\mathcal V_k\) with \(|v|=1\), and choose
\(\phi\in C_c^\infty(\R;\R)\) nonzero and even. For \(\varepsilon>0\), define
the terminal datum by
\[
\widehat q_T^\varepsilon(\xi)
:=
\varepsilon^{-1/2}\phi(\xi/\varepsilon)\,v .
\]
Since \(\phi\) is real-valued and even, and since \(v\in\R^N\), one has
$\widehat q_T^\varepsilon(-\xi)
=
\overline{\widehat q_T^\varepsilon(\xi)}.$
Thus \(q_T^\varepsilon\in L^2(\R;\R^N)\). Moreover, since
\(v\in\mathcal V_k\), one has
$\Pi_k q_T^\varepsilon=q_T^\varepsilon .$
Finally,
$\|q_T^\varepsilon\|_{L^2(\R;\R^N)}
=
\|\phi\|_{L^2(\R)} .$
For \(\varepsilon>0\) small enough, we also have
$\operatorname{supp}\widehat q_T^\varepsilon
\subset
\{|\xi|\le\xi_0\}.$
The Fourier mass of $q_T^\varepsilon$ concentrates near $\xi=0$ as $\varepsilon\to0$.
Let $q^\varepsilon$ be the corresponding solution of the adjoint system. For
$\varepsilon>0$ small enough so that
$\supp \widehat q_T^\varepsilon\subset\{|\xi|\le \xi_0\}$, the preceding estimate yields
\[
\int_0^T\!\!\int_\R |B^\top q^\varepsilon|^2\,dx\,dt
=
\int_\R (\widehat q_T^\varepsilon)^*\mathcal G_T(\xi)\widehat q_T^\varepsilon\,d\xi
\le C_k\,\varepsilon^{2k}
\]
for some constant $C_k>0$ independent of $\varepsilon$.
Assume by contradiction that there exists $C>0$ such that, for all
$q_T\in L^2(\R;\R^N)$ and corresponding solutions $q$,
\[
\int_\R w(\xi)\,\big|\widehat{\Pi_k q_T}(\xi)\big|^2\,d\xi
\le
C\int_0^T\!\!\int_\mathcal O |B^\top q|^2\,dx\,dt.
\]
Since $\mathcal O\subset\R$, one has
\[
\int_0^T\!\!\int_\mathcal O |B^\top q|^2\,dx\,dt
\le
\int_0^T\!\!\int_\R |B^\top q|^2\,dx\,dt.
\]
Hence the same estimate holds with $\R$ in place of $\mathcal O$. Applying it to
$q_T^\varepsilon$ and using $\Pi_k q_T^\varepsilon=q_T^\varepsilon$, we obtain
\[
\int_\R w(\xi)|\widehat q_T^\varepsilon(\xi)|^2\,d\xi
\le C\,C_k\,\varepsilon^{2k}.
\]
Changing variables $\xi=\varepsilon\eta$ yields
\[
\int_\R w(\varepsilon\eta)|\phi(\eta)|^2\,d\eta
\le C\,C_k\,\varepsilon^{2k}.
\]
By the assumption on $w$, for any $M>0$ there exists $\delta_M>0$ such that $w(\xi)\ge M|\xi|^{2k}$ for all $|\xi|\le\delta_M$. Choosing $\varepsilon>0$ small enough so that $\varepsilon|\eta|\le\delta_M$ on $\supp\phi$, we obtain
\[
\int_\R w(\varepsilon\eta)|\phi(\eta)|^2\,d\eta
\ge M\,\varepsilon^{2k}\int_\R |\eta|^{2k}|\phi(\eta)|^2\,d\eta.
\]
Since $\phi\not\equiv0$, the right-hand side is bounded below by $M c_0\varepsilon^{2k}$ for some $c_0>0$, which yields a contradiction for $M$ large enough. This proves that no uniformly stronger low-frequency weight than $|\xi|^{2k}$
can be admissible on the $k$-th Kalman layer.
\end{proof}
\section{Global observability beyond diagonalizability: the Jordan case}
\label{jordann}
The results proved in Sections~2--5 rely on the assumption that \(A\) is diagonalizable over \(\mathbb R\). In this section, we restrict ourselves to the whole-line case $\mathcal O=\mathbb R$ and allow matrices with real spectrum and nontrivial Jordan blocks. We work with the complexification of the system on the Fourier side.

After Fourier transform in space, every frequency \(\xi\) gives rise to a finite-dimensional observation problem. Near \(\xi=0\), the degeneracy of the observation Gramian is still governed by the Kalman filtration of the pair \((A,B)\). Away from the origin, the nilpotent parts of the Jordan decomposition generate polynomial factors in \(|\xi|\), as is standard for weakly hyperbolic systems with multiple characteristics; see, for instance,
\cite{DAnconaSpagnolo1998,GarettoJahRuzhansky2018}.
When several Jordan blocks carry the same eigenvalue \(\lambda\), they all have the same oscillatory phase \(e^{is\xi\lambda}\). Consequently, their observed coefficients at a fixed nilpotent order cannot be separated by oscillations. Their contributions must therefore be summed before the norm is taken. This leads naturally to an aggregated Jordan norm. Throughout this section, assume that
\[
A\in\mathbb R^{N\times N},
\quad
B\in\mathbb R^{N\times m},
\quad
\sigma(A)\subset\mathbb R,
\]
and that the pair \((A,B)\) satisfies the Kalman rank condition
\[
\operatorname{rank}\bigl(B,AB,\ldots,A^{N-1}B\bigr)=N.
\]
\subsection{A Sobolev estimate with derivative loss for non-diagonalizable systems}
Let \(m_{\max}\) denote the maximal size of the Jordan blocks of \(A\), and set $\nu:=m_{\max}-1.$
\begin{proposition}
\label{prop:WP-jordan}
For every \(s\in\mathbb R\) and every \(T>0\), there exists \(C_T>0\) such that
\[
\bigl\|e^{-tA\partial_x}y_0\bigr\|_{H^s(\mathbb R)^N}
\leq
C_T\|y_0\|_{H^{s+\nu}(\mathbb R)^N}
\]
for every \(t\in[0,T]\) and every
\(y_0\in H^{s+\nu}(\mathbb R)^N\).
\end{proposition}
\begin{proof}
Choose a real Jordan decomposition $A=PJ_AP^{-1}.$
For a Jordan block
\[
J_r(\lambda)=\lambda I_r+N_r,
\quad
N_r^r=0,
\]
we have
\[
e^{-it\xi J_r(\lambda)}
=
e^{-it\xi\lambda}
\sum_{j=0}^{r-1}
\frac{(-it\xi)^j}{j!}N_r^j.
\]
Since \(\lambda\in\mathbb R\), the oscillatory factor has modulus one.
Hence, uniformly for \(t\in[0,T]\),
\[
\|e^{-it\xi J_r(\lambda)}z\|
\leq
C_{r,T}(1+|\xi|)^{r-1}\|z\|.
\]
Taking the maximum over the finitely many Jordan blocks and returning to
the original coordinates gives
\[
\|e^{-it\xi A}z\|
\leq
C_{A,T}(1+|\xi|)^\nu\|z\|.
\]
The conclusion follows from Plancherel:
\[
\begin{aligned}
\bigl\|e^{-tA\partial_x}y_0\bigr\|_{H^s}^2
&\leq
C_T
\int_{\mathbb R}
(1+|\xi|^2)^s
(1+|\xi|)^{2\nu}
|\widehat y_0(\xi)|^2\,d\xi
\\
&\leq
C_T
\int_{\mathbb R}
(1+|\xi|^2)^{s+\nu}
|\widehat y_0(\xi)|^2\,d\xi.
\end{aligned}
\]
\end{proof}

\begin{remark}
Proposition~\ref{prop:WP-jordan} only records the maximal Sobolev loss
generated by the longest Jordan chain. The observability space introduced
below is sharper: it retains the different weights associated with the
successive nilpotent orders.
\end{remark}

\subsection{Spectral–nilpotent observation coefficients}
Let $\Lambda:=\sigma(A)$
be the set of distinct eigenvalues. For every \(\lambda\in\Lambda\), let
\(m_\lambda\) be the size of the largest Jordan block associated with
\(\lambda\), and define the full generalized eigenspace
\[
\mathcal E_\lambda
:=
\ker\bigl((A^\top-\lambda I)^{m_\lambda}\bigr).
\]
Let
\[
\mathsf P_\lambda:\mathbb C^N\longrightarrow\mathcal E_\lambda
\]
denote the corresponding generalized spectral projector. Then
\[
\mathbb C^N
=
\bigoplus_{\lambda\in\Lambda}\mathcal E_\lambda,
\qquad
\sum_{\lambda\in\Lambda}\mathsf P_\lambda=I.
\]

For \(z\in\mathbb C^N\) and
\(0\leq j\leq m_\lambda-1\), define the spectral–nilpotent observation coefficient
\begin{equation}
\label{eq:aggregated-Jordan-coefficient}
\mathfrak a_{\lambda,j}(z)
:=
\frac{1}{j!}
B^\top(A^\top-\lambda I)^j\mathsf P_\lambda z
\in\mathbb C^m.
\end{equation}
This definition is intrinsic and does not depend on the choice of a Jordan basis.
To make the aggregation explicit, choose a real Jordan decomposition
\[
A^\top=SJS^{-1},
\qquad
\widetilde B^\top:=B^\top S,
\]
and group together all the Jordan blocks associated with the same
eigenvalue. For each \(\lambda\in\Lambda\), write
\[
J_\lambda
=
\operatorname{diag}\bigl(
J_{m_{\lambda,1}}(\lambda),
\ldots,
J_{m_{\lambda,r_\lambda}}(\lambda)
\bigr),
\quad
J_{m_{\lambda,r}}(\lambda)
=
\lambda I_{m_{\lambda,r}}+N_{\lambda,r},
\]
where \(r_\lambda\) is the number of Jordan blocks associated with
\(\lambda\), \(m_{\lambda,r}\) is the size of the \(r\)-th block, and
$m_\lambda
=
\max_{1\leq r\leq r_\lambda}m_{\lambda,r}.$ If
\[
u=S^{-1}z,
\qquad
u_\lambda
=
\begin{pmatrix}
u_{\lambda,1}\\
\vdots\\
u_{\lambda,r_\lambda}
\end{pmatrix},
\]
and if the part of \(\widetilde B^\top\) corresponding to the generalized
eigenspace associated with \(\lambda\) is decomposed as
\[
\widetilde B_\lambda^\top
=
\begin{pmatrix}
\widetilde B_{\lambda,1}^\top
&
\cdots
&
\widetilde B_{\lambda,r_\lambda}^\top
\end{pmatrix},
\]
then
\[
\mathfrak a_{\lambda,j}(z)
=
\frac{1}{j!}
\sum_{r=1}^{r_\lambda}
\widetilde B_{\lambda,r}^\top
N_{\lambda,r}^j u_{\lambda,r}.
\]
Here and below, \(N_{\lambda,r}^j=0\) whenever
\(j\geq m_{\lambda,r}\).
Thus, for a fixed eigenvalue \(\lambda\) and a fixed nilpotent order \(j\),
the sum over the Jordan blocks is taken before the norm.

\begin{lemma}
\label{lem:Jordan-jet-nondegeneracy}
Assume that \((A,B)\) satisfies the Kalman rank condition. For every
\(\lambda\in\Lambda\), the map
\[
\mathcal E_\lambda
\longrightarrow
(\mathbb C^m)^{m_\lambda},
\qquad
z\longmapsto
\left(
B^\top(A^\top-\lambda I)^jz
\right)_{0\leq j\leq m_\lambda-1},
\]
is injective. Consequently, there exists \(\beta_\lambda>0\) such that
\begin{equation}
\label{eq:Jordan-jet-lambda-coercivity}
\sum_{j=0}^{m_\lambda-1}
\left\|
B^\top(A^\top-\lambda I)^jz
\right\|_{\mathbb C^m}^{2}
\geq
\beta_\lambda\|z\|_{\mathbb C^N}^{2},
\qquad
z\in\mathcal E_\lambda.
\end{equation}

Moreover, there exists \(\beta>0\) such that
\begin{equation}
\label{eq:aggregated-Jordan-nondegeneracy}
\sum_{\lambda\in\Lambda}
\sum_{j=0}^{m_\lambda-1}
\|\mathfrak a_{\lambda,j}(z)\|_{\mathbb C^m}^{2}
\geq
\beta\|z\|_{\mathbb C^N}^{2},
\qquad
z\in\mathbb C^N.
\end{equation}
\end{lemma}

\begin{proof}
Fix \(\lambda\in\Lambda\), and let \(z\in\mathcal E_\lambda\) satisfy
\[
B^\top(A^\top-\lambda I)^jz=0,
\qquad
0\leq j\leq m_\lambda-1.
\]
Since $(A^\top-\lambda I)^{m_\lambda}z=0,$
the binomial formula gives
\[
(A^\top)^kz
=
\sum_{j=0}^{\min\{k,m_\lambda-1\}}
\binom{k}{j}
\lambda^{k-j}
(A^\top-\lambda I)^jz.
\]
Hence
\[
B^\top(A^\top)^kz=0,
\qquad
k\geq0.
\]
By the dual form of the Kalman rank condition,
\[
\bigcap_{k=0}^{N-1}
\ker\bigl(B^\top(A^\top)^k\bigr)
=
\{0\},
\]
we conclude that \(z=0\). Estimate
\eqref{eq:Jordan-jet-lambda-coercivity} follows by compactness on the unit sphere of \(\mathcal E_\lambda\).
Now assume that
$\mathfrak a_{\lambda,j}(z)=0$
for every \(\lambda\in\Lambda\) and every
\(0\leq j\leq m_\lambda-1\). Applying the first part to
\(\mathsf P_\lambda z\in\mathcal E_\lambda\) gives
$\mathsf P_\lambda z=0$
for every \(\lambda\). Therefore \(z=0\). The quantitative estimate
\eqref{eq:aggregated-Jordan-nondegeneracy} follows again by
finite-dimensional compactness.
\end{proof}

For every \(\xi\neq0\), define
\begin{equation}
\label{eq:aggregated-Jordan-norm}
\|z\|_{\mathcal J,\xi}^{2}
:=
\sum_{\lambda\in\Lambda}
\sum_{j=0}^{m_\lambda-1}
|\xi|^{2j}
\|\mathfrak a_{\lambda,j}(z)\|_{\mathbb C^m}^{2}.
\end{equation}
By Lemma~\ref{lem:Jordan-jet-nondegeneracy}, this is a norm on
\(\mathbb C^N\) for every \(\xi\neq0\).

\subsection{Uniform coercivity away from zero frequency}

Let \(K\) be the Kalman index of \((A,B)\), and let
\(\Pi_0,\ldots,\Pi_K\) be the orthogonal projectors associated with the Kalman-layer decomposition introduced in the previous sections. Set
\[
\|z\|_{\mathcal K,\xi}^{2}
:=
\sum_{k=0}^{K}
|\xi|^{2k}\|\Pi_kz\|_{\mathbb C^N}^{2}.
\]
Fix once and for all a threshold
$\bar\xi>0,$
independently of the observation time \(T\). We define
\(\mathcal Y_{K,J}^{-}\) as the Hilbert completion of
\(\mathcal S(\mathbb R;\mathbb R^N)\) for the norm
\begin{equation}
\label{eq:Jordan-observation-space}
\begin{aligned}
\|q_T\|_{\mathcal Y_{K,J}^{-}}^{2}
:={}&
\int_{|\xi|\leq\bar\xi}
\|\widehat q_T(\xi)\|_{\mathcal K,\xi}^{2}\,d\xi
+
\int_{|\xi|>\bar\xi}
\|\widehat q_T(\xi)\|_{\mathcal J,\xi}^{2}\,d\xi.
\end{aligned}
\end{equation}
Thus, only two regimes enter the definition of the space. The
intermediate-frequency region is treated in the proof by compactness and does not require a separate component in the norm.
\subsection{Jordan coercivity away from the origin}
For \(T>0\), define the observation Gramian
\begin{equation}
\label{eq:Jordan-observation-Gramian}
\mathcal G_T(\xi)
:=
\int_0^T
e^{-is\xi A}BB^\top e^{is\xi A^\top}\,ds.
\end{equation}
Then, for every \(z\in\mathbb C^N\),
\[
z^*\mathcal G_T(\xi)z
=
\int_0^T
\left\|
B^\top e^{is\xi A^\top}z
\right\|_{\mathbb C^m}^{2}\,ds.
\]

\begin{lemma}
\label{lem:Jordan-away-zero}
Assume that \((A,B)\) satisfies the Kalman rank condition. For every
\(T>0\) and every \(\xi_*>0\), there exist constants
\(c_{T,\xi_*},C_{T,\xi_*}>0\) such that
\begin{equation}
\label{eq:Jordan-away-zero}
c_{T,\xi_*}\|z\|_{\mathcal J,\xi}^{2}
\leq
z^*\mathcal G_T(\xi)z
\leq
C_{T,\xi_*}\|z\|_{\mathcal J,\xi}^{2}
\end{equation}
for every \(z\in\mathbb C^N\) and every \(|\xi|\geq\xi_*\).
\end{lemma}

\begin{proof}
Since
$z=\sum_{\lambda\in\Lambda}\mathsf P_\lambda z$
and
\[
e^{is\xi A^\top}\mathsf P_\lambda z
=
e^{is\xi\lambda}
\sum_{j=0}^{m_\lambda-1}
\frac{(is\xi)^j}{j!}
(A^\top-\lambda I)^j\mathsf P_\lambda z,
\]
we obtain
\begin{equation}
\label{eq:intrinsic-Jordan-expansion}
B^\top e^{is\xi A^\top}z
=
\sum_{\lambda\in\Lambda}
e^{is\xi\lambda}
\sum_{j=0}^{m_\lambda-1}
(is\xi)^j\mathfrak a_{\lambda,j}(z).
\end{equation}

For each \(\lambda\in\Lambda\), set
\[
Q_\lambda(s,\xi)
:=
\sum_{j=0}^{m_\lambda-1}
(i\xi)^js^j\mathfrak a_{\lambda,j}(z),
\qquad
F_\lambda(s,\xi)
:=
e^{is\xi\lambda}Q_\lambda(s,\xi).
\]
Then
\[
B^\top e^{is\xi A^\top}z
=
\sum_{\lambda\in\Lambda}F_\lambda(s,\xi).
\]

By the finite expansion above and the Cauchy--Schwarz inequality,
\[
\left\|
B^\top e^{is\xi A^\top}z
\right\|_{\mathbb C^m}^{2}
\leq
C_T
\sum_{\lambda\in\Lambda}
\sum_{j=0}^{m_\lambda-1}
|\xi|^{2j}
\|\mathfrak a_{\lambda,j}(z)\|_{\mathbb C^m}^{2},
\]
uniformly for \(s\in[0,T]\). Hence
\[
z^*\mathcal G_T(\xi)z
\leq
C_T\|z\|_{\mathcal J,\xi}^{2}.
\]
For the lower bound, the polynomial Gramian estimate \eqref{gramm} gives,
for every \(\lambda\in\Lambda\),
\begin{equation}
\label{eq:Jordan-diagonal-estimate}
\int_0^T
\|F_\lambda(s,\xi)\|_{\mathbb C^m}^{2}\,ds
=
\int_0^T
\|Q_\lambda(s,\xi)\|_{\mathbb C^m}^{2}\,ds
\geq
c_{\mathrm{pol},T}
\sum_{j=0}^{m_\lambda-1}
|\xi|^{2j}
\|\mathfrak a_{\lambda,j}(z)\|_{\mathbb C^m}^{2}.
\end{equation}

Fix \(\lambda\neq\mu\), and set
\[
H_{\lambda,\mu}(s,\xi)
:=
\langle Q_\lambda(s,\xi),Q_\mu(s,\xi)\rangle_{\mathbb C^m}.
\]
Then
\[
\langle F_\lambda(s,\xi),F_\mu(s,\xi)\rangle_{\mathbb C^m}
=
e^{is\xi(\mu-\lambda)}H_{\lambda,\mu}(s,\xi).
\]
An integration by parts yields
\[
\begin{aligned}
\int_0^T
e^{is\xi(\mu-\lambda)}
H_{\lambda,\mu}(s,\xi)\,ds
={}&
\frac{
e^{iT\xi(\mu-\lambda)}H_{\lambda,\mu}(T,\xi)
-
H_{\lambda,\mu}(0,\xi)
}{
i\xi(\mu-\lambda)
}
-
\frac{1}{i\xi(\mu-\lambda)}
\int_0^T
e^{is\xi(\mu-\lambda)}
\partial_sH_{\lambda,\mu}(s,\xi)\,ds.
\end{aligned}
\]
Since
\[
\partial_sQ_\lambda(s,\xi)
=
\sum_{j=1}^{m_\lambda-1}
j(i\xi)^js^{j-1}\mathfrak a_{\lambda,j}(z),
\]
we have, uniformly for \(s\in[0,T]\),
\[
\|Q_\lambda(s,\xi)\|_{\mathbb C^m}^{2}
+
\|\partial_sQ_\lambda(s,\xi)\|_{\mathbb C^m}^{2}
\leq
C_T
\sum_{j=0}^{m_\lambda-1}
|\xi|^{2j}
\|\mathfrak a_{\lambda,j}(z)\|_{\mathbb C^m}^{2}.
\]
Moreover,
$\partial_sH_{\lambda,\mu}
=
\langle \partial_sQ_\lambda,Q_\mu\rangle_{\mathbb C^m}
+
\langle Q_\lambda,\partial_sQ_\mu\rangle_{\mathbb C^m}.$
Therefore, by the Cauchy--Schwarz and Young inequalities, uniformly for
\(s\in[0,T]\),
\[
|H_{\lambda,\mu}(s,\xi)|
+
|\partial_sH_{\lambda,\mu}(s,\xi)|
\leq
C_T
\left(
\sum_{j=0}^{m_\lambda-1}
|\xi|^{2j}\|\mathfrak a_{\lambda,j}(z)\|_{\mathbb C^m}^{2}
+
\sum_{j=0}^{m_\mu-1}
|\xi|^{2j}\|\mathfrak a_{\mu,j}(z)\|_{\mathbb C^m}^{2}
\right).
\]
If \(\#\Lambda=1\), there are no cross terms, and \eqref{eq:Jordan-diagonal-estimate} directly yields
$z^*\mathcal G_T(\xi)z
\geq
c_{\mathrm{pol},T}\|z\|_{\mathcal J,\xi}^{2}.$
Assume henceforth that \(\#\Lambda\geq2\), and set
\[
\delta_\Lambda
:=
\min_{\substack{\lambda,\mu\in\Lambda\\ \lambda\neq\mu}}
|\lambda-\mu|>0.
\]
Then there exists \(C_{\mathrm{cross},T,A}>0\) such that, for
\(|\xi|\geq1\),
\[
\begin{aligned}
\left|
\int_0^T
\langle F_\lambda(s,\xi),F_\mu(s,\xi)\rangle_{\mathbb C^m}\,ds
\right|
\leq
\frac{C_{\mathrm{cross},T,A}}{|\xi|}
\Bigg(
&
\sum_{j=0}^{m_\lambda-1}
|\xi|^{2j}\|\mathfrak a_{\lambda,j}(z)\|_{\mathbb C^m}^{2}
+
\sum_{j=0}^{m_\mu-1}
|\xi|^{2j}\|\mathfrak a_{\mu,j}(z)\|_{\mathbb C^m}^{2}
\Bigg).
\end{aligned}
\]
Using
\[
\left\|
\sum_{\lambda\in\Lambda}F_\lambda
\right\|_{\mathbb C^m}^{2}
=
\sum_{\lambda\in\Lambda}\|F_\lambda\|_{\mathbb C^m}^{2}
+
2\operatorname{Re}
\sum_{\lambda<\mu}
\langle F_\lambda,F_\mu\rangle_{\mathbb C^m},
\]
and summing over the finitely many pairs, we obtain
\[
z^*\mathcal G_T(\xi)z
\geq
\left(
c_{\mathrm{pol},T}
-
\frac{C_{\mathrm{cross},T,A}}{|\xi|}
\right)
\|z\|_{\mathcal J,\xi}^{2}.
\]

Choose
\[
R_{T,\xi_*}
\geq
\max\left\{
1,\xi_*,
\frac{2C_{\mathrm{cross},T,A}}{c_{\mathrm{pol},T}}
\right\}.
\]
Then
$z^*\mathcal G_T(\xi)z
\geq
\frac{c_{\mathrm{pol},T}}{2}
\|z\|_{\mathcal J,\xi}^{2}$
for every \(|\xi|\geq R_{T,\xi_*}\).
It remains to consider $\xi_*\leq |\xi|\leq R_{T,\xi_*}.$
As proved in Section~3, the Gramian \(\mathcal G_T(\xi)\) is positive
definite for every \(\xi\neq0\). Moreover, by
Lemma~\ref{lem:Jordan-jet-nondegeneracy},
\(\|\cdot\|_{\mathcal J,\xi}\) is a norm for every \(\xi\neq0\).
Therefore,
\[
\Phi(\xi,z)
:=
\frac{z^*\mathcal G_T(\xi)z}
     {\|z\|_{\mathcal J,\xi}^{2}}
\]
is continuous and strictly positive on the compact set
\[
\left\{
(\xi,z)\in\mathbb R\times\mathbb C^N:
\xi_*\leq|\xi|\leq R_{T,\xi_*},
\quad
\|z\|_{\mathbb C^N}=1
\right\}.
\]
Consequently, \(\Phi\) has a positive minimum on this set. By homogeneity,
there exists \(c_{T,\xi_*}>0\) such that
\[
z^*\mathcal G_T(\xi)z
\geq
c_{T,\xi_*}\|z\|_{\mathcal J,\xi}^{2}
\]
for every \(z\in\mathbb C^N\) and every
\(\xi_*\leq|\xi|\leq R_{T,\xi_*}\).

Combining this estimate with the large-frequency bound and the upper
estimate completes the proof.
\end{proof}
\subsection{Global observability}
\begin{proposition}
\label{prop:obs-jordan}
Assume that \(\sigma(A)\subset\mathbb R\) and that \((A,B)\) satisfies the
Kalman rank condition. Then, for every \(T>0\), there exist constants
\(c_T,C_T>0\) such that every smooth terminal datum \(q_T\) and the
corresponding solution of \eqref{adjoint-system} satisfy
\begin{equation}
\label{eq:two-sided-Jordan-observation}
c_T\|q_T\|_{\mathcal Y_{K,J}^{-}}^{2}
\leq
\int_0^T\int_{\mathbb R}
|B^\top q(t,x)|^2\,dx\,dt
\leq
C_T\|q_T\|_{\mathcal Y_{K,J}^{-}}^{2}.
\end{equation}
Consequently, the observation map extends continuously from
\(\mathcal Y_{K,J}^{-}\) into
\(L^2((0,T)\times\mathbb R)^m\), and
\begin{equation}
\label{eq:Jordan-global-observability}
\|q_T\|_{\mathcal Y_{K,J}^{-}}^{2}
\leq
C_T
\int_0^T\int_{\mathbb R}
|B^\top q(t,x)|^2\,dx\,dt
\end{equation}
holds for every \(q_T\in\mathcal Y_{K,J}^{-}\).
\end{proposition}
\begin{proof}
Writing \(s=T-t\), we have
$\widehat q(t,\xi)
=
e^{is\xi A^\top}\widehat q_T(\xi).$
Hence, by Plancherel's theorem,
\begin{equation}
\label{eq:J-obs-Plancherel}
\int_0^T\int_{\mathbb R}
|B^\top q(t,x)|^2\,dx\,dt
=
\int_{\mathbb R}
\widehat q_T(\xi)^*
\mathcal G_T(\xi)
\widehat q_T(\xi)\,d\xi.
\end{equation}

We now combine the fibrewise estimates established above. The
low-frequency argument of Section~3 provides \(\delta_T>0\) such that
\[
z^*\mathcal G_T(\xi)z
\asymp_T
\|z\|_{\mathcal K,\xi}^2,
\qquad |\xi|\leq\delta_T.
\]
Moreover, on every compact annulus away from the origin, the positive
definiteness and continuity of \(\mathcal G_T(\xi)\) give uniform
equivalence with the Euclidean norm. Finally,
Lemma~\ref{lem:Jordan-away-zero} gives
\[
z^*\mathcal G_T(\xi)z
\asymp_{T,\bar\xi}
\|z\|_{\mathcal J,\xi}^2,
\qquad |\xi|\geq\bar\xi.
\]

Set \(\delta=\min\{\delta_T,\bar\xi/2\}\). On the compact annulus
\(\delta\leq|\xi|\leq\bar\xi\), both
\(\|\cdot\|_{\mathcal K,\xi}\) and
\(\|\cdot\|_{\mathcal J,\xi}\) are uniformly equivalent to the
Euclidean norm. For the Jordan norm, this follows from
Lemma~\ref{lem:Jordan-jet-nondegeneracy}, continuity, and compactness.
Consequently, after adjusting the constants,
\begin{equation}
\label{eq:J-global-fibre}
z^*\mathcal G_T(\xi)z
\asymp_{T,\bar\xi}
\begin{cases}
\|z\|_{\mathcal K,\xi}^2,
& |\xi|\leq\bar\xi,\\[1ex]
\|z\|_{\mathcal J,\xi}^2,
& |\xi|>\bar\xi.
\end{cases}
\end{equation}

Integrating this estimate in \(\xi\) and using
\eqref{eq:J-obs-Plancherel}, we obtain
\[
\int_0^T\int_{\mathbb R}
|B^\top q(t,x)|^2\,dx\,dt
\asymp_{T,\bar\xi}
\|q_T\|_{\mathcal Y_{K,J}^{-}}^2.
\]
Since \(\bar\xi\) is fixed, its dependence may be absorbed into the
constants. This proves \eqref{eq:two-sided-Jordan-observation} for
Schwartz terminal data.

The upper estimate extends the observation map continuously to
\(\mathcal Y_{K,J}^{-}\), while the lower estimate passes to the limit
by density. This proves \eqref{eq:Jordan-global-observability}.
\end{proof}
\subsection{Controllability and the dual space}

\begin{corollary}
\label{cor:jordan-controllability}
Assume that \(\sigma(A)\subset\mathbb R\) and that \((A,B)\) satisfies the
Kalman rank condition. Let
\[
\mathcal Y_{K,J}^{+}
:=
(\mathcal Y_{K,J}^{-})'
\]
be the Hilbert dual of \(\mathcal Y_{K,J}^{-}\) with respect to the
\(L^2\)-Fourier pairing. Then, for every \(T>0\), the system
\[
y_t+Ay_x=Bu,
\qquad
u\in L^2((0,T)\times\mathbb R;\mathbb R^m),
\]
is exactly controllable, equivalently null controllable, in
\(\mathcal Y_{K,J}^{+}\).
\end{corollary}

\begin{proof}
Proposition~\ref{prop:obs-jordan} gives the continuity and coercivity of
the adjoint observation map on \(\mathcal Y_{K,J}^{-}\), while
Lemma~\ref{lem:jordan-reversibility} gives continuity and
reversibility of the homogeneous dynamics. The abstract HUM equivalence
of Proposition~\ref{prop:abstract-HUM} therefore applies with
\[
X=\mathcal Y_{K,J}^{+},
\qquad
X'=\mathcal Y_{K,J}^{-}.
\]
The observability estimate
\eqref{eq:Jordan-global-observability} yields exact controllability. Since
the homogeneous dynamics is reversible, exact and null controllability are
equivalent.
\end{proof}

We finally describe the dual norm more explicitly. For
\(\lambda\in\Lambda\) and
\(0\leq j\leq m_\lambda-1\), define
\[
\mathscr A_{\lambda,j}
:=
\frac{1}{j!}
B^\top(A^\top-\lambda I)^j\mathsf P_\lambda
:
\mathbb C^N\longrightarrow\mathbb C^m.
\]
Then
\[
\mathfrak a_{\lambda,j}(z)
=
\mathscr A_{\lambda,j}z
\quad
\text{and}
\quad
\|z\|_{\mathcal J,\xi}^{2}
=
\sum_{\lambda\in\Lambda}
\sum_{j=0}^{m_\lambda-1}
|\xi|^{2j}
\|\mathscr A_{\lambda,j}z\|_{\mathbb C^m}^{2}.
\]

For \(\xi\neq0\), the dual Jordan norm is
\begin{equation}
\label{eq:dual-Jordan-norm}
\|w\|_{\mathcal J,\xi,*}^{2}
:=
\inf
\left\{
\sum_{\lambda\in\Lambda}
\sum_{j=0}^{m_\lambda-1}
|\xi|^{-2j}
\|\eta_{\lambda,j}\|_{\mathbb C^m}^{2}
\ ;\
w=
\sum_{\lambda\in\Lambda}
\sum_{j=0}^{m_\lambda-1}
\mathscr A_{\lambda,j}^{*}\eta_{\lambda,j}
\right\}.
\end{equation}
The infimum is finite because the injectivity established in
Lemma~\ref{lem:Jordan-jet-nondegeneracy} implies the surjectivity of the
corresponding adjoint map.

The \(L^2\)-pivot dual can therefore be realized as the weighted Fourier
space
\begin{equation}
\label{eq:Jordan-control-space}
\begin{aligned}
\|\varphi_T\|_{\mathcal Y_{K,J}^{+}}^{2}
:={}&
\int_{|\xi|\leq\bar\xi}
\sum_{k=0}^{K}
|\xi|^{-2k}
\|\Pi_k\widehat\varphi_T(\xi)\|_{\mathbb C^N}^{2}\,d\xi+
\int_{|\xi|>\bar\xi}
\|\widehat\varphi_T(\xi)\|_{\mathcal J,\xi,*}^{2}\,d\xi.
\end{aligned}
\end{equation}
Thus, the low-frequency part is dual to the Kalman-layer degeneracy, whereas the frequencies away from the origin are measured by the dual of the aggregated Jordan norm. In particular, when several Jordan blocks carry the same eigenvalue, the dual variables are indexed by the distinct eigenvalue and the nilpotent order, rather than by the individual blocks.

\begin{remark}
The preceding argument is qualitative with respect to the observation time \(T\). The low-frequency neighborhood in which the Taylor-Kalman argument applies and the large-frequency region in which the interactions between distinct eigenvalues can be absorbed may both depend on \(T\).
The remaining frequencies are handled by compactness, which does not give a lower bound uniform as \(T\to0\). The polynomial Gramian associated with the maximal nilpotent order
$\nu=m_{\max}-1$ suggests a contribution of order \(T^{-(2\nu+1)}\), whereas the low-frequency Kalman mechanism suggests a contribution of order \(T^{-(2K+1)}\). However, proving an estimate of the form
\[
C_{\mathrm{obs},J}(T)
\leq
C
\max
\left\{
T^{-1},
T^{-(2K+1)},
T^{-(2\nu+1)}
\right\}
\]
requires an analysis uniform in the joint parameter \(T\xi\). We do not claim such a quantitative estimate here.
\end{remark}
\subsection{A physical Jordan example: a third-order pressureless moment closure}
We consider the closed third-order moment system
\[
\partial_t M+\partial_xF(M)=0,
\quad
M=(m_0,m_1,m_2)^\top,
\quad
\text{with}
\quad 
F(M)=
\left(
m_1,\,
m_2,\,
\frac{3m_1m_2}{m_0}-\frac{2m_1^3}{m_0^2}
\right)^\top.
\]
This closure is compatible with the monokinetic manifold
$(m_0,m_1,m_2)
=
(\rho,\rho u,\rho u^2),$
on which the third flux is equal to \(\rho u^3\).
Linearizing around
\[
M_0=(\rho_0,\rho_0\lambda,\rho_0\lambda^2),
\qquad
\rho_0>0,
\]
gives
\[
U_t+AU_x=0,
\qquad
A=
\begin{pmatrix}
0&1&0\\
0&0&1\\
\lambda^3&-3\lambda^2&3\lambda
\end{pmatrix}.
\]
The characteristic polynomial is \((\mu-\lambda)^3\), and
\(A-\lambda I\) is nilpotent of order three. Hence \(A^\top\) is similar to the single Jordan block
\[
J=\lambda I+N,
\qquad
N=
\begin{pmatrix}
0&1&0\\
0&0&1\\
0&0&0
\end{pmatrix}.
\]
More precisely, with
\[
S=
\big(
(\lambda^2,-2\lambda,1)^\top,\,
(-\lambda,1,0)^\top,\,
(1,0,0)^\top
\big),
\]
one has $S^{-1}A^\top S=J.$
We take \(B=e_3\), so that the adjoint observation is the third physical component. Since $B^\top S=(1,0,0),$
the observation in Jordan coordinates is
\[
\widetilde B^\top z=z_1,
\quad
\widetilde B^\top:=B^\top S.
\quad 
\text{Moreover,}
\quad
\rank(B,AB,A^2B)=3,
\]
so the Kalman condition holds with index \(K=2\).
Since \(N^3=0\),
\[
\widetilde B^\top e^{is\xi J}z
=
e^{is\xi\lambda}
\left(
z_1+is\xi z_2-\frac{s^2\xi^2}{2}z_3
\right).
\]
The polynomial Gramian associated with \(1,s,s^2\) therefore gives
\begin{equation}
\label{eq:physical-Jordan-Gramian}
\int_0^T
\left|
\widetilde B^\top e^{is\xi J}z
\right|^2\,ds
\asymp_T
|z_1|^2
+
|\xi|^2|z_2|^2
+
\frac{|\xi|^4}{4}|z_3|^2,
\qquad
\xi\in\mathbb R.
\end{equation}
In these coordinates, the Kalman filtration is represented successively
by the directions \(e_1,e_2,e_3\).

Here there is a single eigenvalue and a single Jordan block, so that
\[
r_\lambda=1,
\qquad
m_\lambda=3,
\qquad
\mathsf P_\lambda=I.
\]
The aggregated Jordan coefficients therefore reduce to
\[
\mathfrak a_{\lambda,0}(Sz)=z_1,
\qquad
\mathfrak a_{\lambda,1}(Sz)=z_2,
\qquad
\mathfrak a_{\lambda,2}(Sz)=\frac{z_3}{2}.
\]
Consequently,
\[
\|Sz\|_{\mathcal J,\xi}^{2}
=
|z_1|^2
+
|\xi|^2|z_2|^2
+
\frac{|\xi|^4}{4}|z_3|^2.
\]

Thus, in this single-block example, the low-frequency Kalman weights and the Jordan-jet weights involve the same powers of \(|\xi|\). Their origins are nevertheless different: near the origin they reflect the Kalman filtration, whereas away from the origin they arise from the successive nilpotent orders of the Jordan block.
\section{Conclusion and perspectives}
\label{sec:conclusion}
This work shows that low-frequency observability losses in partially observed
transport systems are governed by the Kalman filtration of the pair $(A,B)$.
The finite-time observation Gramian generates the corresponding anisotropic
Fourier scale, which provides the natural functional framework for whole-line
observability and exact controllability.

In the diagonalizable case, the Kalman filtration determines the sharp
low-frequency weights on each visibility layer. Combined with
Logvinenko--Sereda localization estimates and characteristic propagation, this
yields localized observability on thick sets above the sharp geometric control
threshold.

For real-spectrum systems with nontrivial Jordan structure, the same
low-frequency Kalman mechanism persists. The additional feature is the
high-frequency polynomial growth along Jordan chains, leading to a
Jordan-adapted observability scale in the whole-line setting.

Localized observability for systems with nontrivial Jordan structure lies beyond the scope of the present paper. Although the low-frequency Kalman mechanism persists, components within the same Jordan block propagate with the same characteristic speed and remain coupled through the nilpotent part. Consequently, the diagonal separation argument used in the localized proof does not extend directly to this setting. A possible approach would combine triangular recovery along Jordan chains with the high-frequency weights of the Jordan-adapted scale, while accounting for the interaction between the Kalman projectors and the nilpotent parts.

Other natural directions include lower-order couplings, variable-coefficient
systems, boundary effects, higher-dimensional transport geometries, and a
microlocal interpretation of the Kalman--Gramian scale. These extensions require
understanding how the layerwise Kalman structure interacts with perturbative,
geometric, and microlocal effects beyond the constant-coefficient
one-dimensional setting.
\section*{Acknowledgments}

The author is deeply grateful to Enrique Zuazua for suggesting the original problem and for the fundamental insight that led to the main results of this work. He also thanks him for his constant guidance, insightful discussions, and encouragement throughout the development of the paper.

The author is grateful to Haitian Yang from Tsinghua University and to Lorenzo Liverani from Friedrich-Alexander-Universit\"at Erlangen--N\"urnberg for their helpful comments.

The author acknowledges that this work was carried out during a research stay at the Chair for Dynamics, Control, Machine Learning and Numerics, Friedrich-Alexander-Universit\"at Erlangen--N\"urnberg, and was supported by the Alexander von Humboldt Foundation through an Alexander von Humboldt
Fellowship.
\section*{Statements and Declarations}
\paragraph{Funding.}
This work was supported by the Alexander von Humboldt Foundation through
an Alexander von Humboldt Fellowship.
\paragraph{Competing interests.}
The author has no relevant financial or non-financial interests to
disclose.
\paragraph{Data availability.}
No datasets were generated or analysed during the current study.
\paragraph{Code availability.}
Not applicable.
\clearpage
\appendix
\section{Duality, evolution and stability properties of the Kalman-adapted spaces}
For the remainder of this appendix, we assume that
$A\in\mathbb R^{N\times N},
\quad
B\in\mathbb R^{N\times m},$
and that the pair \((A,B)\) satisfies the Kalman rank condition.
All the state, control, and dual spaces are regarded as real Hilbert
spaces. Their Fourier representations are complex-valued and satisfy
the usual conjugate-symmetry condition.
\label{app:HUM}

This appendix records the abstract HUM principle used in the paper and verifies
that it applies to the anisotropic spaces introduced above. We use the dual
spaces
$\mathcal Y_K^{+}:=(\mathcal Y_K^{-})',
\quad
\mathcal Y_{K,J}^{+}:=(\mathcal Y_{K,J}^{-})'.$
Thus the relevant adjoint/primal pairs are
\[
(\mathcal Y_K^{-},\mathcal Y_K^{+})
\quad\text{in the diagonalizable case,}
\quad
(\mathcal Y_{K,J}^{-},\mathcal Y_{K,J}^{+})
\quad\text{in the real-spectrum non-diagonalizable whole-line case.}
\]
The HUM argument is standard: an observability estimate for the adjoint equation
in the negative space yields exact controllability of the primal equation in the
corresponding dual space. Thus the global and localized observability estimates
in $\mathcal Y_K^{-}$ imply controllability in $\mathcal Y_K^{+}$, while the
global Jordan estimate in $\mathcal Y_{K,J}^{-}$ gives controllability in
$\mathcal Y_{K,J}^{+}$. 
\begin{proposition}
\label{prop:abstract-HUM}
Let $X$ be a Hilbert space, let $X'$ be its dual, and let
$\mathcal O\subset\mathbb R$ be the observation set. Set
$U:=L^2((0,T)\times\mathcal O;\mathbb R^m).$
Assume that the homogeneous primal evolution $S(T)$ is an isomorphism on $X$.
Assume moreover that, for every $q_T\in X'$, the adjoint solution with terminal datum $q_T$ is well defined and that the observation map
\[
O_T:X'\longrightarrow U,
\qquad
O_Tq_T:=B^\top q,
\]
is bounded. Define $W_T\in\mathcal L(U,X)$ by transposition:
\[
\langle W_T\Theta,q_T\rangle_{X,X'}
=
(\Theta,O_Tq_T)_U
=
\int_0^T\!\!\int_{\mathcal O}
\Theta(t,x)\cdot B^\top q(t,x)\,dx\,dt.
\]
Then the following properties are equivalent:
\begin{enumerate}
\item the primal system is exactly controllable in $X$ at time $T$;
\item the primal system is null controllable in $X$ at time $T$;
\item there exists $C_T>0$ such that every adjoint solution satisfies
\[
\|q_T\|_{X'}^2
\le
C_T
\int_0^T\!\!\int_{\mathcal O}|B^\top q(t,x)|^2\,dx\,dt .
\]
\end{enumerate}
\end{proposition}
\begin{proof}
This is the standard HUM duality argument. By transposition,
\(W_T' q_T=O_Tq_T\). Moreover
\[
y(T)=S(T)y_0+W_T\Theta .
\]
Hence exact controllability is equivalent to \(\operatorname{Ran}W_T=X\). Null controllability is equivalent to \(\operatorname{Ran}S(T)\subset \operatorname{Ran}W_T\), and since \(S(T)\) is onto, this is again equivalent to
\(\operatorname{Ran}W_T=X\). Finally, by the closed range theorem,
\[
\operatorname{Ran}W_T=X
\quad\Longleftrightarrow\quad
\exists C>0,\quad
\|q_T\|_{X'}\le C\|W_T' q_T\|_U
\quad \forall q_T\in X'.
\]
Since \(W_T'q_T=O_Tq_T=B^\top q\), this is precisely the observability
inequality.
\end{proof}
The following lemma provides the bounded reversible evolution of the homogeneous dynamics, which will be used in the proof of Proposition \ref{prop:abstract-HUM}.
\begin{lemma}\label{bounded}
Assume that $A$ is diagonalizable over $\mathbb R$ with real eigenvalues. Then the homogeneous adjoint dynamics associated with \eqref{adjoint-system} defines a bounded reversible evolution on $\mathcal Y_K^-$ over $[0,T]$, and the homogeneous primal dynamics defines a bounded reversible evolution on $\mathcal Y_{K}^+$ over $[0,T]$. More precisely, there exists $C_T>0$ such that for every $t\in[0,T]$
\[
\|q(t)\|_{\mathcal Y_K^-}\le C_T\|q_T\|_{\mathcal Y_K^-},
\qquad
\|q_T\|_{\mathcal Y_K^-}\le C_T\|q(t)\|_{\mathcal Y_K^-}, \quad \text{and}
\]
\[
\|y(t)\|_{\mathcal Y_{K}^+}\le C_T\|y_0\|_{\mathcal Y_{K}^+},
\qquad
\|y_0\|_{\mathcal Y_{K}^+}\le C_T\|y(t)\|_{\mathcal Y_{K}^+}.
\]
\end{lemma}
\begin{proof}
We treat first the adjoint evolution. In Fourier variables, $\widehat q(t,\xi)=e^{i\xi(T-t)A^\top}\widehat q_T(\xi).$
Let
\[
\mathcal K_k=\operatorname{Ran}(B,AB,\dots,A^kB),\qquad \mathcal V_k=\mathcal K_k\cap \mathcal K_{k-1}^\perp,
\]
and let $\Pi _k$ be the orthogonal projector onto $\mathcal V_k$.
Since $A(\mathcal K_{k-1})\subset \mathcal K_k$, one has by duality
$A^\top(\mathcal K_k^\perp)\subset \mathcal K_{k-1}^\perp.$
Iterating this inclusion yields
$ \Pi_k (A^\top)^m \Pi_j=0
\quad\text{for } m<j-k,\quad k<j.$
This is the triangular relation. It quantifies, at low frequency, transfers
from deeper Kalman layers to shallower ones. The off-diagonal block from $\mathcal V_j$ to $\mathcal V_k$, with
$k<j$, carries the factor $|\xi|^{j-k}$, which is exactly what is needed to
match the anisotropic weights defining $\mathcal Y_K^-$ and, by duality,
$\mathcal Y_K^+$.
Set
$F_{kj}(t,\xi):=\Pi_k e^{i\xi(T-t)A^\top}\Pi_j.$
Since $A$ is diagonalizable over $\mathbb R$ with real spectrum,
\[
\sup_{t\in[0,T],\,\xi\in\R}\|e^{i\xi(T-t)A^\top}\|\le C_T.
\]
Moreover, for $k<j$, the triangular relation implies that all $\xi$-derivatives
of $F_{kj}(t,\xi)$ up to order $j-k-1$ vanish at $\xi=0$.
Using Taylor's formula for $|\xi|\le 1$, uniformly in $t\in[0,T]$, we obtain
\[
\|F_{kj}(t,\xi)\|\le C_T |\xi|^{j-k},\qquad k<j,
\]
while for $|\xi|\ge 1$, or if $k\ge j$, the uniform bound gives
$\|F_{kj}(t,\xi)\|\le C_T.$
Hence, for all $j,k$,
\[
\|F_{kj}(t,\xi)\|\le C_T\min(|\xi|^{(j-k)_+},1).
\]
With
$\rho_k(\xi):=\min(|\xi|^{2k},1),$
it follows that
\[
\rho_k(\xi)\,|F_{kj}(t,\xi)v|^2
\le C_T\, \rho_j(\xi)\,|v|^2.
\quad
\text{Since}
\quad 
\Pi_k\widehat q(t,\xi)=\sum_{j=0}^K F_{kj}(t,\xi)\Pi_j\widehat q_T(\xi),
\]
and the number of blocks is finite, we infer
\[
\sum_{k=0}^K \rho_k(\xi)|\Pi_k\widehat q(t,\xi)|^2
\le
C_T\sum_{j=0}^K \rho_j(\xi)|\Pi_j\widehat q_T(\xi)|^2.
\]
Integrating in $\xi$ yields
$\|q(t)\|_{\mathcal Y_K^-}\le C_T\|q_T\|_{\mathcal Y_K^-}.$
Since the inverse evolution is given by $e^{-i\xi(T-t)A^\top}$, the same argument
applies verbatim and yields
$\|q_T\|_{\mathcal Y_K^-}\le C_T\|q(t)\|_{\mathcal Y_K^-}.$
We now turn to the primal evolution. In Fourier variables,
$\widehat y(t,\xi)=e^{-i\xi tA}\widehat y_0(\xi).$
Using again $A(\mathcal K_{k-1})\subset \mathcal K_k$, one obtains the primal triangular relation
\[
\Pi_k A^m \Pi_j=0
\qquad\text{for } m<k-j,\quad k>j.
\]
Arguing as above, but now with the reciprocal weights defining $\mathcal Y_K^+$, we obtain
\[
\|y(t)\|_{\mathcal Y_K^+}\le C_T\|y_0\|_{\mathcal Y_K^+}.
\]
Since the system is linear and autonomous, the same estimate holds for the inverse evolution, hence
\[
\|y_0\|_{\mathcal Y_K^+}\le C_T\|y(t)\|_{\mathcal Y_K^+}.
\]

This shows that both the homogeneous adjoint and primal dynamics define bounded
reversible evolutions on $\mathcal Y_K^-$ and $\mathcal Y_K^+$, respectively.
\end{proof}
We prove that the homogeneous dynamics is continuous and reversible on the
Jordan-adapted observation space introduced in Section~6.

\begin{lemma}
\label{lem:jordan-reversibility}
Assume that \(A^\top\) has real spectrum, and let
\(\mathcal Y_{K,J}^{-}\) be defined as in Section~6. Then, for every
\(T>0\), there exists \(C_T>0\) such that every solution of the homogeneous
adjoint system satisfies
\[
C_T^{-1}\|q_T\|_{\mathcal Y_{K,J}^{-}}
\leq
\|q(t)\|_{\mathcal Y_{K,J}^{-}}
\leq
C_T\|q_T\|_{\mathcal Y_{K,J}^{-}},
\qquad
t\in[0,T].
\]
Moreover, the homogeneous adjoint dynamics defines a strongly continuous
and reversible evolution on \(\mathcal Y_{K,J}^{-}\). Consequently, the
homogeneous primal dynamics defines a strongly continuous and reversible
evolution on the dual space
$\mathcal Y_{K,J}^{+}
:=
(\mathcal Y_{K,J}^{-})'.$
\end{lemma}
\begin{proof}
For the homogeneous adjoint solution, one has $\widehat q(t,\xi)
=
e^{i\xi(T-t)A^\top}\widehat q_T(\xi).$
We argue separately in the two frequency regimes entering the definition of \(\mathcal Y_{K,J}^{-}\). On the low-frequency region $|\xi|\leq\bar\xi,$ the triangular Kalman relations imply that
\[
\Pi_k(A^\top)^r\Pi_j=0,
\quad
r<j-k,\quad k<j.
\]
Taylor's formula therefore gives
\[
\left\|
\Pi_k e^{i\xi(T-t)A^\top}\Pi_j
\right\|
\leq
C_T|\xi|^{(j-k)_+},
\quad
|\xi|\leq\bar\xi,\quad t\in[0,T].
\]
It follows that
\[
\sum_{k=0}^{K}
|\xi|^{2k}
\left|
\Pi_k\widehat q(t,\xi)
\right|^2
\leq
C_T
\sum_{j=0}^{K}
|\xi|^{2j}
\left|
\Pi_j\widehat q_T(\xi)
\right|^2.
\]
Applying the same estimate to the inverse evolution gives the reverse inequality. Hence
\[
\|\widehat q(t,\xi)\|_{\mathcal K,\xi}
\asymp_T
\|\widehat q_T(\xi)\|_{\mathcal K,\xi},
\quad
|\xi|\leq\bar\xi.
\]
We now consider the region $|\xi|>\bar\xi.$
For every \(\lambda\in\Lambda\), the generalized spectral projector
\(\mathsf P_\lambda\) commutes with \(A^\top\). Hence, for
\(0\leq j\leq m_\lambda-1\),
\[
\begin{aligned}
\mathfrak a_{\lambda,j}
\bigl(\widehat q(t,\xi)\bigr)
&=
\mathfrak a_{\lambda,j}
\left(
e^{i\xi(T-t)A^\top}\widehat q_T(\xi)
\right)=
e^{i\xi(T-t)\lambda}
\sum_{r=0}^{m_\lambda-1-j}
\binom{j+r}{j}
\bigl(i\xi(T-t)\bigr)^r
\mathfrak a_{\lambda,j+r}
\bigl(\widehat q_T(\xi)\bigr).
\end{aligned}
\]
Consequently,
\[
|\xi|^j
\left\|
\mathfrak a_{\lambda,j}
\bigl(\widehat q(t,\xi)\bigr)
\right\|
\leq
C_T
\sum_{r=0}^{m_\lambda-1-j}
|\xi|^{j+r}
\left\|
\mathfrak a_{\lambda,j+r}
\bigl(\widehat q_T(\xi)\bigr)
\right\|.
\]
Since only finitely many eigenvalues and nilpotent orders occur, summing
over \(j\) and \(\lambda\) gives
\[
\|\widehat q(t,\xi)\|_{\mathcal J,\xi}
\leq
C_T
\|\widehat q_T(\xi)\|_{\mathcal J,\xi},
\qquad
|\xi|>\bar\xi.
\]
Applying the same estimate to the inverse evolution gives
\[
\|\widehat q_T(\xi)\|_{\mathcal J,\xi}
\leq
C_T
\|\widehat q(t,\xi)\|_{\mathcal J,\xi}.
\quad
\text{Therefore,}
\quad 
\|\widehat q(t,\xi)\|_{\mathcal J,\xi}
\asymp_T
\|\widehat q_T(\xi)\|_{\mathcal J,\xi},
\qquad
|\xi|>\bar\xi.
\]

Integrating the low- and high-frequency estimates with respect to \(\xi\)
yields
\[
\|q(t)\|_{\mathcal Y_{K,J}^{-}}
\asymp_T
\|q_T\|_{\mathcal Y_{K,J}^{-}},
\qquad
t\in[0,T].
\]
For terminal data in
\(\mathcal S(\mathbb R;\mathbb C^N)\), the continuity of
$t\longmapsto q(t)$ in \(\mathcal Y_{K,J}^{-}\) follows from dominated convergence. The locally uniform estimates above and the density of \(\mathcal S(\mathbb R;\mathbb C^N)\) extend the homogeneous dynamics to a strongly continuous and reversible evolution on \(\mathcal Y_{K,J}^{-}\).
The corresponding statement for the homogeneous primal dynamics follows by Hilbert duality.
\end{proof}
\section{Observation-induced norm equivalences and admissibility estimates}
\label{app:norm-equivalence}
\begin{proof}[Proof of Corollary~\ref{equivalent-norm}]
The lower bound is precisely the localized observability estimate
\eqref{eq:local}. We only prove the upper bound.

Since \(\mathcal O\subset\mathbb R\), Plancherel's theorem and the change of
variables \(s=T-t\) give
\[
\int_0^T\int_{\mathcal O}|B^\top q(t,x)|^2\,dx\,dt
\leq
C\int_{\mathbb R}\int_0^T
\left|
B^\top e^{is\xi A^\top}\widehat q_T(\xi)
\right|^2\,ds\,d\xi ,
\]
where the harmless Plancherel constant is absorbed into \(C\).

We claim that, for every \(k=0,\ldots,K\), there exists \(\mathcal C_T>0\) such that
\[
\left\|B^\top e^{is\xi A^\top}\Pi_k\right\|
\leq
\mathcal C_T\rho_k(\xi)^{1/2},
\qquad
0\leq s\leq T,\quad \xi\in\mathbb R,
\]
where $\rho_k(\xi)=\min\{|\xi|^{2k},1\}.$
Indeed, for \(|\xi|\leq1\), the Kalman-layer cancellation
\[
B^\top(A^\top)^j\Pi_k=0,
\qquad j<k,
\]
and Taylor's formula yield
\[
B^\top e^{is\xi A^\top}\Pi_k
=
O_T(|\xi|^k),
\qquad 0\leq s\leq T .
\]
Thus
\[
\left\|B^\top e^{is\xi A^\top}\Pi_k\right\|
\leq
\mathcal C_T|\xi|^k,
\qquad |\xi|\leq1.
\]
For \(|\xi|\geq1\), since \(A\) is diagonalizable over \(\mathbb R\), the
group \(e^{is\xi A^\top}\) is uniformly bounded for \(0\leq s\leq T\) and
\(\xi\in\mathbb R\). Hence
\[
\left\|B^\top e^{is\xi A^\top}\Pi_k\right\|
\leq
\mathcal C_T,
\qquad |\xi|\geq1.
\]
Combining the two estimates gives the claim.

Now let \(z\in\mathbb C^N\). Since \(z=\sum_{k=0}^K\Pi_k z\), the previous
estimate and the finiteness of the number of Kalman layers imply
\[
\int_0^T
\left|
B^\top e^{is\xi A^\top}z
\right|^2\,ds
\leq
\mathcal C_T
\sum_{k=0}^K
\rho_k(\xi)|\Pi_k z|^2 .
\]
Applying this with \(z=\widehat q_T(\xi)\) and integrating in \(\xi\), we get
\[
\int_0^T\int_{\mathcal O}|B^\top q(t,x)|^2\,dx\,dt
\leq
\mathcal C_T
\sum_{k=0}^K
\int_{\mathbb R}
\rho_k(\xi)|\Pi_k\widehat q_T(\xi)|^2\,d\xi .
\]
Therefore,
\[
\int_0^T\int_{\mathcal O}|B^\top q(t,x)|^2\,dx\,dt
\leq
\mathcal C_T\|q_T\|_{\mathcal Y_K^-}^2 .
\]
This proves the upper bound and completes the proof of the norm equivalence.
\end{proof}
We also record the corresponding upper bound in the Jordan-adapted whole-line
scale.
\begin{lemma}\label{lem:jordan-upper}
For every $T>0$, there exists $C_T>0$ such that every smooth terminal datum
$q_T$ satisfies
\[
\int_0^T\int_{\mathbb R}|B^\top q(t,x)|^2\,dx\,dt
\leq
C_T
\|q_T\|_{\mathcal Y_{K,J}^{-}}^2,
\]
where $q$ solves \eqref{adjoint-system}. Consequently, the observation map
$q_T\mapsto B^\top q$ extends continuously from $\mathcal Y_{K,J}^{-}$ into
$L^2((0,T)\times\mathbb R)^m$.
\end{lemma}
\begin{proof}
By Plancherel's theorem, with \(s=T-t\),
\[
\int_0^T\int_{\mathbb R}
|B^\top q(t,x)|^2\,dx\,dt
=
\int_{\mathbb R}\int_0^T
\left|
B^\top e^{is\xi A^\top}\widehat q_T(\xi)
\right|^2\,ds\,d\xi.
\]

We distinguish the two frequency regimes entering the definition of
\(\mathcal Y_{K,J}^{-}\).

For \(|\xi|\leq\bar\xi\), the Kalman-layer cancellations $B^\top(A^\top)^j\Pi_k=0,
\qquad
j<k,$
together with Taylor's formula give
\[
\int_0^T
\left|
B^\top e^{is\xi A^\top}\widehat q_T(\xi)
\right|^2\,ds
\leq
C_T
\sum_{k=0}^{K}
|\xi|^{2k}
\left|
\Pi_k\widehat q_T(\xi)
\right|^2.
\]

For \(|\xi|>\bar\xi\), the intrinsic Jordan expansion
\[
B^\top e^{is\xi A^\top}z
=
\sum_{\lambda\in\Lambda}
e^{is\xi\lambda}
\sum_{j=0}^{m_\lambda-1}
(is\xi)^j
\mathfrak a_{\lambda,j}(z)
\]
and the finiteness of the sums imply
\[
\int_0^T
\left|
B^\top e^{is\xi A^\top}z
\right|^2\,ds
\leq
C_T
\sum_{\lambda\in\Lambda}
\sum_{j=0}^{m_\lambda-1}
|\xi|^{2j}
\|\mathfrak a_{\lambda,j}(z)\|_{\mathbb C^m}^{2}
=
C_T
\|z\|_{\mathcal J,\xi}^{2}.
\]

Applying the latter estimate with
\(z=\widehat q_T(\xi)\), integrating over the two frequency regions, and
using the definition of \(\mathcal Y_{K,J}^{-}\), we obtain
\[
\int_0^T\int_{\mathbb R}
|B^\top q(t,x)|^2\,dx\,dt
\leq
C_T
\|q_T\|_{\mathcal Y_{K,J}^{-}}^{2}.
\]
The continuous extension of the observation map follows by density.
\end{proof}
\section{Characteristic recovery on thick observation sets} \label{lemm}
\begin{proof}
Fix \(T_1>T_\mu^*\), and let \(T\in(T_\mu^*,T_1]\). Set
\[
\tau:=T-T_\mu^*>0,\qquad
T_\mu^*:=\frac{G_{\mathcal O}}{|\mu|},\qquad
U_T:=(0,T)\times\mathcal O,
\]
and choose
\[
\varepsilon:=\frac{\tau}{8},
\qquad
\delta:=\frac1{16}\min\{\ell_{\mathcal O},|\mu|\tau\}.
\]
Let
\[
\mathcal O_\delta
:=
\bigcup_{j\in\mathbb Z}(a_j+\delta,b_j-\delta),
\qquad
\ell_\delta:=\ell_{\mathcal O}-2\delta,
\qquad
G_\delta:=G_{\mathcal O}+2\delta .
\]
Then
\[
2\delta\le \frac{\ell_{\mathcal O}}8,
\qquad
2\delta\le \frac{|\mu|\tau}{8},
\qquad
\ell_\delta\ge \frac78\ell_{\mathcal O},
\qquad
\ell_\delta+G_\delta=\ell_{\mathcal O}+G_{\mathcal O},
\]
and
\[
|\mu|(T-2\varepsilon)-G_\delta
=
|\mu|\tau-\frac{|\mu|\tau}{4}-2\delta
\ge
\frac58|\mu|\tau .
\]

Choose \(\chi\in C^\infty(\mathbb R)\) and
\(\theta\in C_c^\infty(0,T)\) such that
\[
0\le \chi\le \mathbf 1_{\mathcal O},
\qquad
\chi=1\ \text{on }\mathcal O_\delta,
\qquad
\operatorname{supp}\chi\subset\mathcal O_{\delta/2},
\]
and
\[
0\le\theta\le1,
\qquad
\theta=1\ \text{on }[\varepsilon,T-\varepsilon].
\]
The cutoffs are chosen with fixed profiles. Since \(0<\tau\le T_1\), the
definition of \(\delta\) gives, for every \(j\ge0\),
\[
\delta^{-j}\le C_j(T_1,\mu,\ell_{\mathcal O})\tau^{-j}.
\]
Consequently,
\begin{equation}
\label{eq:cutoff-uniform-T1}
\|\chi^{(j)}\|_{L^\infty}
\le C_j(T_1,\mu,\ell_{\mathcal O})\tau^{-j},
\qquad
\|\theta^{(j)}\|_{L^\infty}
\le C_j\tau^{-j}.
\end{equation}
Moreover, for \(j\ge1\), the support of \(\theta^{(j)}\) is contained in two
intervals of length \(O(\varepsilon)\), hence
\[
|\operatorname{supp}\theta^{(j)}|\le C\tau .
\]

Define
\[
m(y):=
\int_0^T
\theta(t)\chi(y-\mu(T-t))\,dt.
\]
Since \(\theta=1\) on \([\varepsilon,T-\varepsilon]\) and
\(\chi=1\) on \(\mathcal O_\delta\), the change of variables
\(x=y-\mu(T-t)\), together with Lemma~\ref{lem:density-bands} applied to
\(\mathcal O_\delta\), gives
\[
m(y)
\ge
\frac1{|\mu|}
|\mathcal O_\delta\cap I_y|,
\qquad
|I_y|=|\mu|(T-2\varepsilon).
\]
Therefore,
\[
|\mathcal O_\delta\cap I_y|
\ge
\frac{\ell_\delta}{\ell_\delta+G_\delta}
\bigl(|I_y|-G_\delta\bigr)_+
\ge
\frac78
\frac{\ell_{\mathcal O}}{\ell_{\mathcal O}+G_{\mathcal O}}
\frac58|\mu|\tau.
\]
Thus
\begin{equation}
\label{eq:m-lower-T1}
m(y)\ge c_0\tau,
\qquad
c_0:=
\frac{35}{64}
\frac{\ell_{\mathcal O}}{\ell_{\mathcal O}+G_{\mathcal O}}.
\end{equation}

Set
\[
\eta(t,y):=\theta(t)\chi(y-\mu(T-t)),
\qquad
b(t,y):=\frac{\eta(t,y)}{m(y)}.
\]
Then
\[
m(y)=\int_0^T\eta(t,y)\,dt,
\qquad
\int_0^T b(t,y)\,dt=1.
\]

We next record the uniform bounds on \(b\). Let
\[
N_{T_1}:=2+\frac{|\mu|T_1}{\ell_{\mathcal O}}.
\]
Indeed, if a spatial interval of length at most \(|\mu|T\le |\mu|T_1\)
meets \(n\) components of \(\mathcal O\), then
\[
(n-2)_+\ell_{\mathcal O}\le |\mu|T_1,
\]
and hence \(n\le N_{T_1}\). Therefore the time spent by a characteristic in
the boundary layers of \(\chi\) is bounded by
\[
\frac{N_{T_1}\delta}{|\mu|}
\le C(T_1,\mu,\ell_{\mathcal O})\tau.
\]

Let \(\alpha=(\alpha_t,\alpha_y)\). For \(|\alpha|\ge1\), every term in
\(\partial_t^{\alpha_t}\partial_y^{\alpha_y}\eta\) contains a derivative of
\(\theta\) or a derivative of \(\chi\). Using \eqref{eq:cutoff-uniform-T1}
and the support estimates above, we obtain
\begin{equation}
\label{eq:eta-fibre-uniform-T1}
\sup_{y\in\mathbb R}
\left\|
\partial_t^{\alpha_t}\partial_y^{\alpha_y}
\eta(\cdot,y)
\right\|_{L^2(0,T)}
\le
C_\alpha(T_1,\mu,\ell_{\mathcal O})
\tau^{\frac12-|\alpha|}.
\end{equation}
For \(\alpha=0\), since \(0\le \eta\le1\),
\[
\|\eta(\cdot,y)\|_{L^2(0,T)}^2
\le
\int_0^T\eta(t,y)\,dt
=
m(y).
\]

For \(j\ge1\),
\[
m^{(j)}(y)
=
\int_0^T
\theta(t)\chi^{(j)}(y-\mu(T-t))\,dt.
\]
The integrand is supported only when the characteristic meets the boundary
layers of \(\chi\). Hence, using the previous bound on the time spent in
these layers and \eqref{eq:cutoff-uniform-T1},
\[
|m^{(j)}(y)|
\le
\|\chi^{(j)}\|_{L^\infty}
\frac{N_{T_1}\delta}{|\mu|}
\le
C_j(T_1,\mu,\ell_{\mathcal O})\tau^{1-j}.
\]
Together with \(m(y)\ge c_0\tau\), this gives
\[
|m^{(j)}(y)|
\le
C_j(T_1,\mu,\ell_{\mathcal O},G_{\mathcal O})
\tau^{-j}m(y),
\qquad j\ge1.
\]
By Faà di Bruno's formula, for every \(j\ge0\),
\begin{equation}
\label{eq:minv-uniform-T1}
\left|(m^{-1})^{(j)}(y)\right|
\le
C_j(T_1,\mu,\ell_{\mathcal O},G_{\mathcal O})
\tau^{-j}m(y)^{-1}.
\end{equation}

Combining \eqref{eq:eta-fibre-uniform-T1}, \eqref{eq:minv-uniform-T1},
Leibniz' formula, and \(m(y)\ge c_0\tau\), we obtain, for every multi-index
\(\alpha\),
\begin{equation}
\label{eq:b-fibre-uniform-T1}
\sup_{y\in\mathbb R}
\left\|
\partial_t^{\alpha_t}\partial_y^{\alpha_y} b(\cdot,y)
\right\|_{L^2(0,T)}
\le
C_\alpha(T_1,\mu,\ell_{\mathcal O},G_{\mathcal O})
\tau^{-|\alpha|-\frac12}.
\end{equation}
When no derivative falls on \(\eta\), we use
\[
\|\eta(\cdot,y)\|_{L^2(0,T)}\le m(y)^{1/2}.
\]
In particular,
\[
\|b(\cdot,y)\|_{L^2(0,T)}
\le
m(y)^{-1/2}
\le
c_0^{-1/2}\tau^{-1/2}.
\]
We first prove the \(L^2\)-estimate. For \(g\in L^2(\mathbb R)^d\), set
$M(y):=
\int_0^T
\mathbf 1_{\mathcal O}(y-\mu(T-t))\,dt .$
Then
\[
\int_0^T\int_{\mathcal O}
|g(x+\mu(T-t))|^2\,dx\,dt
=
\int_{\mathbb R}|g(y)|^2M(y)\,dy.
\]
Since \(M(y)\ge m(y)\), \eqref{eq:m-lower-T1} gives
\[
\int_0^T\int_{\mathcal O}
|g(x+\mu(T-t))|^2\,dx\,dt
\ge
c_0\tau \|g\|_{L^2(\mathbb R)^d}^2.
\]
Hence
\[
\|g\|_{L^2(\mathbb R)^d}^2
\le
\frac{c_0^{-1}}{T-T_\mu^*}
\int_0^T\int_{\mathcal O}
|g(x+\mu(T-t))|^2\,dx\,dt,
\]
which proves \eqref{eq:L2-characteristic-recovery}.

It remains to prove the \(H^{-s}\)-estimate. For
\(\Psi\in C_c^\infty(U_T)^d\), define
\[
R_\mu\Psi(y):=
\int_0^T
\Psi(t,y-\mu(T-t))\,dt.
\]
Then
\begin{equation}
\label{eq:Rmu-uniform-T1}
\|R_\mu\Psi\|_{H^s(\mathbb R)^d}
\le
C_s T_1^{1/2}\|\Psi\|_{H_0^s(U_T)^d}.
\end{equation}
Indeed, for \(0\le k\le s\), Cauchy's inequality gives
\[
\|\partial_y^kR_\mu\Psi\|_{L^2(\mathbb R)^d}^2
\le
T
\|\partial_x^k\Psi\|_{L^2(U_T)^d}^2
\le
T_1
\|\partial_x^k\Psi\|_{L^2(U_T)^d}^2.
\]

For \(g\in H^{-s}(\mathbb R)^d\), define
\[
\left\langle
g(x+\mu(T-t)),\Psi
\right\rangle_{U_T}
:=
\left\langle
g,R_\mu\Psi
\right\rangle_{\mathbb R}.
\]
This is well-defined by \eqref{eq:Rmu-uniform-T1}.
For \(\phi\in C_c^\infty(\mathbb R)^d\), define
$F_\mu(t,x):=x+\mu(T-t),$
and
\[
(\mathcal E\phi)(t,x)
:=
\theta(t)\chi(x)m(F_\mu(t,x))^{-1}
\phi(F_\mu(t,x)).
\]
Equivalently, in the coordinate \(y=F_\mu(t,x)\),
\[
(\mathcal E\phi)(t,y-\mu(T-t))=b(t,y)\phi(y).
\]
Then \(\mathcal E\phi\in C_c^\infty(U_T)^d\) and
\begin{equation}
\label{eq:right-inverse-uniform-T1}
R_\mu\mathcal E\phi=\phi.
\end{equation}
The affine change of variables \((t,x)\mapsto(t,y)\),
\(y=x+\mu(T-t)\), has Jacobian one, and the Sobolev norms of order \(s\)
in \((t,x)\) and \((t,y)\) are equivalent with a constant depending only on
\(s\) and \(\mu\). Using Leibniz' formula and
\eqref{eq:b-fibre-uniform-T1}, we get
\begin{equation}
\label{eq:E-bound-uniform-T1}
\|\mathcal E\phi\|_{H_0^s(U_T)^d}
\le
C_{\mathcal E,s}(T_1,\mu,\ell_{\mathcal O},G_{\mathcal O})
\tau^{-s-\frac12}
\|\phi\|_{H^s(\mathbb R)^d}.
\end{equation}

Finally, by \eqref{eq:right-inverse-uniform-T1},
\[
\langle g,\phi\rangle_{\mathbb R}
=
\langle g,R_\mu\mathcal E\phi\rangle_{\mathbb R}
=
\left\langle
g(x+\mu(T-t)),\mathcal E\phi
\right\rangle_{U_T}.
\]
Using \eqref{eq:E-bound-uniform-T1}, we obtain
\[
|\langle g,\phi\rangle_{\mathbb R}|
\le
C_{\mathcal E,s}(T_1,\mu,\ell_{\mathcal O},G_{\mathcal O})
\tau^{-s-\frac12}
\|g(x+\mu(T-t))\|_{H^{-s}(U_T)^d}
\|\phi\|_{H^s(\mathbb R)^d}.
\]
Taking the supremum over
\(\phi\in C_c^\infty(\mathbb R)^d\) with
\(\|\phi\|_{H^s(\mathbb R)^d}\le1\), and recalling that
\(\tau=T-T_\mu^*\), gives
\[
\|g\|_{H^{-s}(\mathbb R)^d}
\le
\frac{C_s(T_1)}
{(T-T_\mu^*)^{s+\frac12}}
\|g(x+\mu(T-t))\|_{H^{-s}(U_T)^d}.
\]
This proves \eqref{eq:recovery}
\end{proof}
\bibliographystyle{plain}
\bibliography{biblio}
\vfill 
\end{document}